\documentclass[12pt,reqno]{amsart}
\usepackage{amsmath,amssymb,amsfonts,amsthm}
\usepackage[left=2.5cm, right=2.5cm, top=2.5cm, bottom=2.5cm]{geometry}

\usepackage{pstricks}                        
\usepackage{pstricks-add}               
\usepackage{pst-all,pst-math,pst-plot}       
\usepackage{textcomp}

\usepackage{tikz}
\usetikzlibrary{matrix}

\usepackage{graphicx,color}
\usepackage{graphicx}
\usepackage{amsmath}
\usepackage[all]{xy}
\usepackage{amsmath,amssymb}
\usepackage{amsthm}
\usepackage{amsfonts}
\usepackage{indentfirst}
\usepackage{geometry}
\usepackage[T1]{fontenc}
\usepackage{times}
\usepackage{color,graphicx}

\newtheorem{pr}{Proposition}

\newtheorem{lema}{Lemma}
\newtheorem{theorem}{Theorem}
\newtheorem{de}{Definition}
\newtheorem{remark}{Remark}

\newcommand{\R}{\mathbb{R}}	

\newcommand{\dint}{\displaystyle\int}

\newcommand{\varep}{\varepsilon}	

\begin{document}

	\title[Abstract integrodifferential equations in interpolation scales and applications]{Abstract integrodifferential equations and applications}
	

	\author[B. de Andrade]{Bruno de Andrade}
	\address[B. de Andrade]{Departamento de Matem\'atica, Universidade Federal de Sergipe, S\~ao Crist\'ov\~ao-SE, CEP 49100-000, Brazil.}
	\email{bruno@mat.ufs.br}
	\author[M. G. de Santana]{Marcos Gabriel de Santana}
	\address[M. G. de Santana]{Instituto Federal de Sergipe, Estância-SE, CEP 49200-000, Brazil.}
	\email{marcos.santana@ifs.edu.br}
	\date{}

	\begin{abstract}
	In this work, we study the initial value problem associated with an abstract integro- differential equation in interpolation scales. We prove local-in-time existence, uniqueness, conti- nuation, and a blow-up alternative for regular mild solutions to the problem. Additionally, we apply this theory to the Navier-Stokes equations with hereditary viscosity, taking initial data in the scale of fractional power spaces associated with the Stokes operator. We also explore reaction-diffusion problems with memory, considering the effects of super-linear and gradient-type nonlinearities, and initial data in Lebesgue and Besov spaces, respectively.\\
		
		\
		
		
		
	\end{abstract}

		\maketitle
	

\section{Introduction}

We start this section describing two mathematical models of evolutionary phenomena that motivated us to develop this work. Let $\Omega \subset \R^3$ and consider the Navier-Stokes equation with hereditary viscosity
\begin{equation}\label{ins}\left\{
\begin{array}{lll}
	u_t = \dint_0^t g(t-s)\Delta u(x,s)ds -u\cdot \nabla u -\nabla p + h, \ (x,t) \in \Omega\times(0,\infty), \\
	\mathrm{div}(u)=0 \quad \textrm{in } \Omega\times (0,\infty), \\
	u(x,t)=0, \quad (x,t) \in \partial \Omega\times (0,\infty), \\
	u(x,0)=u_0(x), \quad x \in \Omega.
\end{array}\right.
\end{equation}
Here, $u:\Omega \times (0,\infty) \to \R^3$ describes the velocity field of a fluid contained in $\Omega$, $p$ is the fluid pressure, $h$ is an external force, $u_0:\Omega \to \R^3$ is the velocity field at the initial time $t=0$, and the kernel $g:(0,\infty)\to \R$ represents a material function. This equation arises in the dynamics of non-Newtonian fluids and also as viscoelastic model for the dynamics of turbulence statistics in Newtonian fluids, see \cite{barbu2003navier}. We note that  $g$ plays a central role in the study of \eqref{ins} since homogeneous isotropic incompressible linear materials are described utilizing such a material function, see \cite[Cap. 5]{pruss2013evolutionary}. To the best of our knowledge, the mathematical approach of the  Navier-Stokes equation with hereditary viscosity was considered for the first time in \cite{barbu2003navier}, where using a suitable $m$-accretive quantization of the nonlinear term, Barbu and Sritharan prove a local solvability theorem to the problem in the $L^2(\mathbb{R}^3)$-setting, and a finite speed of propagation property of the vorticity field. Recently, de Andrade, Silva and Viana \cite{andrade2021viana2021silva} consider the above problem to power-type materials, that is, taking
$$1\ast g(t):=\int_{0}^{t} g(s)ds=\frac{t^\alpha}{\Gamma(\alpha+1)},\quad t>0,$$
where $\Gamma$ is the Gamma function and  $\alpha\in[0,1)$. Using complex interpolation and Laplace transform methods, they developed local well-posedness and spatial regularity results to the problem in Lebesgue spaces in the context of a bounded smooth domain $\Omega\subset\mathbb{R}^N$. Particularly, they highlighted how the material function aﬀects the spatial regularity of the mild solutions, namely the solutions have more spatial regularity when $\alpha$ is near to $0$. In \cite{deACD}, de Andrade, Cuevas and Dantas consider \eqref{ins}  with $\Omega=\mathbb{R}^{N}$ ensuring sufficient conditions to the existence and asymptotic stability of mild solutions taking initial data in homogeneous Besov-Morrey spaces. Their analysis is strongly based on a subordination principle with the heat semigroup.

The second motivational problem concerns diffusion phenomena. To problems of heat flow in homogeneous isotropic rigid heat conductors $\Omega\subset \mathbb{R}^N$ subject to hereditary memory, Gurtin and Pipkin \cite{gurtin1968general} proposed to replace the Fourier's constitutive law for the heat flux,
$$q(x,t)=-c \nabla u(x,t),$$
which leads to the classical heat equation, by
\begin{equation*}
	q(x,t)= -\int_0^\infty g(t-s)\nabla u(x,s) ds.
\end{equation*}
This model for the constitutive law for the flux has been well accepted in modeling heat conduction phenomena in materials with memory, see  \cite{monica2014reaction,hristov2016transient,maccamy1977integro,miller1978integrodifferential,nunziato71onheat}, and leads us to equations of the form
\begin{equation}\label{iheat}
	u_t(x,t) =\dint_0^t g(t-s)\Delta u(x,s)\,ds+ f(t, u),\quad (x,t)\in \Omega\times(0, \infty).
\end{equation}
The notion of $C_0$-semigroups can not be used to describe the solutions of \eqref{iheat}, requiring more care in the treatment of such a problem.  Pr\"uss \cite{pruss2013evolutionary} developed an abstract theory of resolvent families that can be applied to \eqref{iheat}. Using that theory, several researchers consider particular situations of \eqref{iheat}, or similar problems, in the most diverse contexts, see  \cite{ andrade2020regularity, andradevianaZAMP, andrade2020gioviana2020viana,Keyantuo2006maximal, lunardi1988, mainini2009, pruss2009, viana2019local, vianaNODEA}. With respect to the nonlinear term of \eqref{iheat}, many publications show that there has been considerable interest in heat conduction in materials with or without memory involving a nonlinear term of gradient or super-linear type. Without being exhaustive, we mention  \cite{almeida2015viana,brezis1996nonlinear,  chen2014semilinear, chipot1989some,  gilding2005cauchy, kaltenbacher2019jordan, loayza2006heat,snoussi2001asymptotically, souplet1996finite, weissler1980,zhang2015blow} and references therein.

From a mathematical point of view, the above problems are very similar in the sense that they can be reformulated as an abstract integrodifferential equation of the form
\begin{equation}\label{abstract equation in introduction}\left\{
	\begin{array}{lll}
	u'(t)=\displaystyle\int_0^tg(t-s)Au(s)ds+f(t,u(t)), \quad t\ge0,
	\\
	u(0)=u_0,
\end{array}\right.
\end{equation}
where
%
$A:\mathcal{D}(A)\subset X_0 \to X_0$ is a closed, unbounded, and densely defined linear operator on a complex Banach space $X_0$ such that $-A$ is sectorial, $g \in L^1_{loc}([0,\infty);\mathbb{C})$ is a non identically zero Laplace transformable function and the nonlinearity $f$ is a suitable function defined in an interpolation scale $\{X_\alpha\}_{\alpha \geq 0}$  associated with $-A$.  For this reason, we will further investigate the correspondent well-posedness and spatial regularity theory to \eqref{abstract equation in introduction} and apply the abstract results to concrete problems; as examples, we study the Navier-Stokes equation with hereditary viscosity in the $L^q$-setting, and the heat equation, derived from the Gurtin-Pipkin law, with a super-linear type nonlinear term in Lebesgue spaces. Additionally, we investigate the heat equation in materials with memory and a gradient type nonlinearity in Besov spaces. To provide a more detailed analysis of our results, we compare them with known results for these equations; moreover, we draw a comparison between our results and previous properties related to some classical problems, this includes the work by Fujita and Kato \cite{fujita1963navier} on the Navier-Stokes equation, the research by Brezis and Cazenave \cite{brezis1996nonlinear} concerning the heat equation, the study by Arrieta and Carvalho \cite{arrieta2000abstract} on abstract semilinear parabolic problems, and the paper by Ben-Artzi, Souplet and Weissler \cite{BSW} on the viscous Hamilton-Jabobi equation.

Going further, as we know, an efficient approach to partial differential equations is closely related to the concept of solutions  to be adopted. With this fact in mind, we are particularly interested in the theory of mild solutions to  \eqref{abstract equation in introduction}, that is,  continuous functions $u:[0,\tau]\to X_1$ verifying
$$u(t)=S(t)u_0 + \int_0^t S(t-s)f(s,u(s))ds, \quad t \in [0,\tau],$$
where $\{S(t)\}_{t \geq 0}$ is the resolvent family associated with the pair $(A,g)$, see Section \ref{section 1} for details. Roughly speaking, our approach of \eqref{abstract equation in introduction} follows the spirit of semilinear parabolic problems within interpolation scales. Indeed, to obtain our results on the integrodifferential problem \eqref{abstract equation in introduction}, we perform the following strategy:
\begin{itemize}
	\item[a)] First, we analyze the behavior of $\{S(t)\}_{t \geq 0}$ in abstract interpolation scales associated with the sectorial operator $-A$.
	\item[b)] Next, we explore these abstract interpolation scales and their connections with some well-established spaces, such as Lebesgue, Besov, and Bessel potential spaces.
	\item[c)] Finally, we examine the nonlinear term $f$ present in \eqref{abstract equation in introduction} within these scales of spaces.
\end{itemize}
Hence, in the next section, we provide a detailed study of the  resolvent family $\{S(t)\}_{t \geq 0}$. Firstly, assuming that $-A$ is a sectorial operator, we provide sufficient conditions on the kernel $g$ to ensure that $\{S(t)\}_{t \geq 0}$ admits analytic extension to a  suitable sector of the complex plane, see Theorem \ref{generation + regularity}.  Also, using sharp interpolations arguments,  we analyze the behavior of $\{S(t)\}_{t \geq 0}$ in $\{X_\alpha\}_{\alpha \geq 0}$. Indeed, assuming the existence of a constant $\zeta_g>1$ such that
$$1/\widehat{g}(\lambda)=O(\lambda^{\zeta_g-1}), \quad \textrm{as }|\lambda|\to \infty,$$
we prove that for any $\tau>0$ and $0\leq \theta\leq\gamma\leq1$, there exists $M>0$ such that
$$\|S(t)x\|_{X_{1+\theta}}\leq Mt^{-\zeta_g(1+\theta-\gamma)}\|x\|_{X_{\gamma}}, \quad t\in(0,\tau],$$
see Theorem \ref{Theorem S from gamma to 1+theta}. 

This smoothing effect is fundamental in our analysis of \eqref{abstract equation in introduction}. In fact, the usual idea to ensure existence of a mild solution is based on the Banach fixed point theorem and consists of showing that the map
$$
T(\psi)(t)=S(t)u_0+\int_0^t S(t-s)f(s,\psi(s))ds,\,\, t \in [0,\tau],$$
is a strict contraction on a suitable complete metric space. If we assume that there exists $\gamma\ge 0$ such that for each $t>0$ the function $f(t,\cdot):X_{1} \to X_{\gamma}$ satisfies 
\begin{equation}\label{fmotiv}
	\|f(t,x)-f(t,y)\|_{X_{\gamma}}\leq c(r)\|x-y\|_{X_{1}}
\end{equation}
for all $x,y \in X_{1}$ with $\|x\|_{X_1}\le r$ and $\|y\|_{X_1}\le r$, then we eventually get integrals as
$$\int_0^t(t-s)^{-\zeta_g(1-\gamma)}ds,$$
whose convergence is equivalent to condition $\zeta_g<\frac{1}{1-\gamma}$.  That is, the technique fails when 
\begin{equation}\label{indicecritico}
	\gamma\le 1-\frac{1}{\zeta_g}.
\end{equation} 
This argumentation underscores that $1-\frac{1}{\zeta_g}$ plays a critical role in the study of well-posedness for problem \eqref{abstract equation in introduction}. Consequently, it seems reasonable to divide the analysis of that problem into two parts according to this value. 


In Section 3, firstly, we treat case $\gamma > 1-\frac{1}{\zeta_g}$ considering initial datum $u_0\in X_1$ and $f(t,\cdot):X_{1} \to X_{\gamma}$ a locally Lipschitz function. We ensure the existence of a unique mild solution $u(\cdot\ ;u_0)\in C([0,\tau ], X_1)$ to \eqref{abstract equation in introduction}, which depends continuously on initial data and can be continued to a maximal time of existence $\tau_{max}>0$ such that
$$\limsup_{t\rightarrow \tau_{max}^{-}}\|u(t;u_0)\|_{X_{1+\theta}}=+\infty,$$
if $\tau_{max}<+\infty$. The fractional powers scale has proven to be an effective tool for investigating the issue of regularity. These spaces offer a precise method for measuring spatial regularity of  solutions. Utilizing them, we demonstrate that the mild solution of \eqref{abstract equation in introduction} exhibits an immediate regularization effect. Indeed, for all $\theta \in (0, \gamma-1+1/\zeta_g)$, we have that
$$u(\cdot\ ;u_0)\in C((0,\tau ],X_{1+\theta}),$$
and, if $J\subset X_1$ is compact, then
$$\lim_{t\to0^{+}}t^{\zeta_g \theta}\sup_{u_0\in J}\|u(t;u_0)\|_{X_{1+\theta}}= 0.$$
Given that $X_{1+\theta}\hookrightarrow X_1$, it follows that $X_{1+\theta}$ is a more regular space than $X_1$. 

Next, we deal with the complementary situation, which will be called the critical case. As we discussed earlier, assuming only that $f(t,\cdot):X_{1} \to X_{\gamma}$  is a locally Lipschitz function, for some $\gamma\le 1-\frac{1}{\zeta_g}$, {it seems to be impossible} to ensure that problem \eqref{abstract equation in introduction} is well-posed in general. To overcome this situation, we use the notion of $\varep$-regular map introduced by Arrieta and Carvalho \cite{arrieta2000abstract}. Basically, taking initial datum $u_0\in X_1$ and assuming the existence of $\varep, \gamma(\varep) \in (0,1)$ such that $f(t,\cdot):X_{1+\varep}\to X_{\gamma(\varep)}$ is well-defined, for each $t>0$, and is a locally Lipschitz function, we obtain local existence, continuous dependence on initial data, continuation, and a blow-up alternative for $\varep$-regular mild solutions of \eqref{abstract equation in introduction}. Also, we prove uniqueness of such solutions in the class of functions with behavior at $t=0$ satisfying the condition\footnote{Conditions of type \eqref{eqws} are usual in integrodifferential equations and arise to address the lack of the semigroup property, see {e.g.} \cite{andrade2016integrodifferential, de2017abstract}.}
\begin{equation}\label{eqws}
	\lim_{t \to 0^+} t^{\zeta_g\varep}\|u(t;u_0)\|_{X_{1+\varep}}=0.
\end{equation} 
Here, $\varep$-regular mild solutions $u(\cdot\ ; u_0)$ are understood as mild solutions that get immediate regularization, in the sense that $u(\cdot\ ; u_0)\in C((0,\tau];X_{1+\varep})$, see Subsection 3.2 for details.  It is important to note that the material function affects the spatial regularity of solutions; specifically, these solutions exhibit greater spatial regularity when $\zeta_g$ is close to $1$, see Theorems \ref{thmsubcritico} and \ref{theorem critic case well-posedness}. This generalizes the particular situation of the Navier-Stokes equation with hereditary viscosity in the context of power-type materials studied in \cite{andrade2021viana2021silva}. To the best of our knowledge, viewed as a general result, this fact is presented for the first
time in this work. 

In Section 4, we apply our abstract results to concrete problems. In such applications, we consider a general scalar kernel of the form
\begin{equation*}
	g(t)=\sum_{i=1}^n k_it^{\alpha_i -1}e^{c_i t}, \quad t>0,
\end{equation*}
where $k_i>0, \alpha_i>0, c_i \in \mathbb{R}.$ We mention that many types of materials, as {Hookean solids}, {Maxwell fluids},  {Poynting-Thompson solids}, and {power-type materials}, are described by this kind of material functions, see \cite{pruss2013evolutionary}. In the Appendix, we demonstrate that these functions satisfy the hypotheses of our theorems when considering certain specific sectorial operators.

\section{Resolvent families and elements of abstract interpolation theory}\sectionmark{Resolvent and interpolation}\label{section 1}

We start this section introducing the notion of {resolvent family} associated with $(A,g)$, which is based on \cite{pruss2013evolutionary}.
\begin{de}[Resolvent family]
	A family $\{S(t)\}_{t \geq 0}$ of bounded linear operators in $X_0$ will be called a {resolvent family} associated with the pair $(A,g)$ if the following conditions are satisfied:
	\begin{itemize}
		\item $S(t)$ is strongly continuous on $[0,\infty)$ and $S(0)=I$;
		\item for each $t\geq 0$ and $x \in \mathcal{D}(A)$, we have $S(t)x \in \mathcal{D}(A)$  and $AS(t)x=S(t)Ax$;
		\item the {resolvent equation}
		\begin{equation}\label{pqo}
			S(t)x=x+\int_0^t\int_0^{t-s}g(t-s-r)drAS(s)xds
		\end{equation}
		holds for all $x \in \mathcal{D}(A)$ and $t \geq 0.$
	\end{itemize}
\end{de}
%
%

We denote by $\Sigma[\omega,\eta]$ the closed sector with vertex $\omega \in \mathbb{R}$ and opening angle $2\eta \in [0,2\pi]$ in the complex plane which is symmetric with respect to the real positive axis, that is,
$$\Sigma[\omega,\eta]:=\{ \omega+re^{i\psi};\, r\geq0, -\eta\leq\psi\leq\eta\}.$$
Further, if $\eta \in (0,\pi)$, we denote $\Sigma(\omega,\eta):=\mathrm{int}\left(\Sigma[\omega,\eta]\right)$, that is,
$$\Sigma(\omega,\eta):=\{ \omega+re^{i\psi};\, r>0, -\eta<\psi<\eta\}.$$
\begin{de}[Analytic resolvent family]
	A resolvent family $\{S(t)\}_{t \geq 0}$ is called {analytic}, if the function $S(\cdot):(0,\infty) \to \mathcal{B}(X_0)$ admits analytic extension to a sector $\Sigma(0,\eta_0)$ for some $0<\eta_0\leq \pi/2$. If that is the case, given $\omega_0\geq 0$, such family is said to be {of analyticity of type $(\omega_0,\eta_0)$} if for each $\omega>\omega_0$ and $\eta \in (0,\eta_0)$ there is $M=M(\omega,\eta)>1$ such that
	$$\|S(z)\|_{\mathcal{B}(X_0)} \leq M e^{\omega \Re(z)}, \quad z \in \Sigma(0, \eta).$$
\end{de}

\begin{de}[Sectorial operator]
	Let $B:\mathcal{D}(B)\subset X_0 \to X_0$ be a closed and unbounded linear operator with dense domain $\mathcal{D}(B)$ and $\psi_0 \in [0,\pi/2)$. We will say that $B$ is {sectorial of angle $\psi_0$}, or simply {sectorial}, if
	\begin{itemize}
		\item the spectrum of $B$, denoted by $\sigma(B)$, is contained in $ \Sigma[0,\psi_0]$;
		\item for each $\psi_1 \in (\psi_0,\pi)$, there is $C=C(\psi_1)$ such that
		$$\left\|(z-B)^{-1}\right\|_{\mathcal{B}(X_0)} \leq C/|z|, \quad z \in \mathbb{C}\backslash\Sigma[0,\psi_1].$$
	\end{itemize}
\end{de}


The next result provides sufficient conditions on the kernel $g$ for the existence of an analytic resolvent family $\{S(t)\}_{t \geq 0}$ associated with the pair $(A,g)$ under the hypothesis of a sectorial operator $-A$. 

\begin{theorem}\label{generation + regularity}
	Let $A:\mathcal{D}(A)\subset X_0 \to X_0$ such that $-A$ is a sectorial operator of angle $\psi_0 \in [0,\pi/2)$ and $g \in L^1_{loc}([0,\infty);\mathbb{C})$ a non identically zero Laplace transformable function. Suppose, for some $\omega_0 \geq 0$ and $\eta_0 \in (0,\pi/2]$, the following conditions hold: \begin{itemize}
		\item[\textrm{(B1)}] $\widehat{g}(\cdot)$ admits meromorphic extension to $\Sigma(\omega_0,\eta_0+\pi/2)$ and $\widehat{g}(\lambda)\neq 0$ for all $\lambda \in \Sigma(\omega_0,\eta_0+\pi/2)$;
		\item[\textrm{(B2)}] For each $\omega_1>\omega_0$ and $\eta_1 \in (0,\eta_0)$, there exists $\psi_1 \in (\psi_0,\pi/2)$ such that
		$$-\lambda / \widehat{g}(\lambda) \in\mathbb{C}\backslash\Sigma[0,\psi_1], \quad \textrm{for all }\lambda \in \Sigma(\omega_1,\eta_1+\pi/2);$$
		\item[\textrm{(B3)}] For some $\omega>\omega_0$ and $\eta \in (0,\eta_0)$, there is $\zeta_g>1$ such that
		$$\limsup_{\substack{|\lambda|\to\infty, \\ \lambda \in \Sigma(\omega,\eta+\pi/2)}}\frac{1}{|\widehat{g}(\lambda)||\lambda|^{\zeta_g-1}}<\infty.$$
	\end{itemize}
	Then, there exists the analytic resolvent family $\{S(t)\}_{t \geq 0}$ associated with the pair $(A,g)$. Furthermore, there is $M\geq 1$ such that
	$$\|S(t)\|_{\mathcal{B}(X_0)} \leq Me^{\omega t}, \quad t \geq 0,$$
	$$\|S(t)\|_{\mathcal{B}(X_0,X_1)} \leq Me^{\omega t}(1+t^{-\zeta_g}), \quad t > 0,$$
	and 
    $$	\|S(t_1)-S(t_0)\|_{\mathcal{B}(X_0)}\leq \frac{Me^{\omega(1+\sin(\eta))t_1}}{\sin(\eta)}\ln\left(t_1/t_0\right), \quad 0<t_0<t_1,$$
	where $X_1=\mathcal{D}(A)$.
\end{theorem}
	\begin{proof}
For $\omega_1>\omega_0$ and $\eta_1 \in (0,\eta_0)$ fixed, let $\psi_1 \in (\psi_0,\pi/2)$ as in (B2). From the sectoriallity of $-A$, there is $C=C(\psi_1)$ such that
		$$\|(z+A)^{-1}\|_{\mathcal{B}(X_0)}\leq C/|z|, \quad z \in \mathbb{C}\backslash\Sigma[0,\psi_1],$$
		from whence, for $\lambda \in \Sigma(\omega_1,\eta_1+\pi/2)$, we have
		\begin{align*}
			\left\|\frac{1}{\widehat{g}(\lambda)}\left(\frac{\lambda}{\widehat{g}(\lambda)}-A\right)^{-1}\right\|_{\mathcal{B}(X_0)}\leq C/|\lambda|\leq  \left(\frac{C}{|\lambda-\omega_1|}\right)\left(1+\omega_1/|\lambda|\right)\leq \frac{C_1}{|\lambda-\omega_1|},
		\end{align*}
		where $C_1 = C(1+1/\cos(\eta_1))$. Then, from  \cite[Theorem 2.1]{pruss2013evolutionary} there exists an analytic resolvent family $\{S(t)\}_{t \geq 0}$ of type $(\omega_0,\eta_0)$ associated with $(A,g)$. In particular, for the constants $\omega>\omega_0$ and $\eta \in (0,\eta_0)$ from (B3), there exists $M=M(\omega,\eta)\geq1$ such that
		\begin{equation*}\label{qpl}
			\|S(t)\|_{\mathcal{B}(X_0)}\leq Me^{\omega t}, \quad
			t \geq 0.
		\end{equation*}
		
		As for the spatial regularity, for each $R>0$, let us denote
		$$\varSigma(R)=\{ \lambda \in \Sigma(\omega,\eta+\pi/2);\, |\lambda|\leq R\}, \quad \varPi(R)=\{ \lambda \in \Sigma(\omega,\eta+\pi/2);\, |\lambda|> R\}$$
		and
		$$m(R)=\sup_{\lambda \in \Pi(R)} \dfrac{|\lambda|}{|\widehat{g}(\lambda)||\lambda-\omega|^{\zeta_g}}.$$
		Since $|\lambda|/|\lambda-\omega| \approx 1$ for $\lambda \in \varPi(R)$ and large $R$, we have
		\begin{align*}
			\lim_{R\to \infty} m(R) = \lim_{R \to \infty} \sup_{\lambda \in \varPi(R)}  \dfrac{1}{|\widehat{g}(\lambda)||\lambda|^{\zeta_g-1}}.
		\end{align*}
		Then, (B3) can be rewritten as $L := \lim_{R \to \infty} m(R) < \infty$. Let $R_0>0$ be large enough so that $m(R_0)\leq L+1$. Then,
		\begin{equation}\label{qxz}
			|\lambda/\widehat{g}(\lambda)| \leq (L+1)|\lambda-\omega|^{\zeta_g}, \quad \textrm{for all }\lambda \in \varPi(R_0).
		\end{equation}
		Since, by (B1), $\widehat{g}(\cdot)$ is a non-vanishing meromorphic function on $\Sigma(\omega_0,\eta_0+\pi/2)$, the map $\lambda\mapsto\lambda/\widehat{g}(\lambda)$ is continuous (even analytic) and then
		\begin{equation}\label{qxc}
			\mu:=\max\left\lbrace|\lambda/\widehat{g}(\lambda)|\,;\, \lambda \in \overline{\varSigma(R_0)}\right\rbrace<\infty.
		\end{equation}
		From \eqref{qxz} and \eqref{qxc}, we have
		$$|\lambda/\widehat{g}(\lambda)|\leq c \left(|\lambda-\omega|^{\zeta_g}+1\right), \quad \textrm{for all }\lambda \in \Sigma(\omega,\eta+\pi/2),$$
		where $c=\max\{L+1,\mu\}$. By  \cite[Theorem 2.2]{pruss2013evolutionary}, there exists $c_0>0$ such that
		$$\|AS(t)\|_{\mathcal{B}(X_0)} \leq c_0 e^{\omega t}\left(1+t^{-\zeta_g}\right), \quad t >0.$$
		Therefore, for $t>0$,
		\begin{align*}
			\|S(t)\|_{\mathcal{B}(X_0,X_1)} &= \|(I-A)S(t)\|_{\mathcal{B}(X_0)} \\
			&\leq \|S(t)\|_{\mathcal{B}(X_0)}+\|AS(t)\|_{\mathcal{B}(X_0)} \\
			&\leq Me^{\omega t}+c_0e^{\omega t}(1+t^{-\zeta_g}) \\
			&\leq (M+c_0)e^{\omega t}(1+t^{-\zeta_g}).
		\end{align*}
		Adjusting the constants, we can rewrite
		$$\|S(t)\|_{\mathcal{B}(X_0)}\leq Me^{\omega t}, \quad t \geq 0,$$
		and
		$$\|S(t)\|_{\mathcal{B}(X_0,X_1)} \leq M e^{\omega t}\left(1+t^{-\zeta_g}\right), \quad t >0.$$

As a direct consequence of \cite[Corollary 2.1]{pruss2013evolutionary}, we have
\begin{equation*}
	\|S(t_1)-S(t_0)\|_{\mathcal{B}(X_0)}\leq \frac{Me^{\omega(1+\sin(\eta))t_1}}{\sin(\eta)}\ln\left(t_1/t_0\right), \quad 0<t_0<t_1,
\end{equation*}		
which concludes the proof.
	\end{proof}

\begin{remark}\label{aer}
	For the sake of simplicity, we can rewrite the inequalities obtained in the Theorem \ref{generation + regularity} as
\begin{equation}
	\|S(t)\|_{\mathcal{B}(X_0)} \leq M, \quad t \in [0,\tau],
\end{equation}
\begin{equation}
	\|S(t)\|_{\mathcal{B}(X_0,X_1)} \leq Mt^{-\zeta_g}, \quad t \in (0,\tau],
\end{equation}
\begin{equation}
	\|S(t_1)-S(t_0)\|_{\mathcal{B}(X_0)}\leq M\ln\left(t_1/t_0\right), \quad 0<t_0<t_1\leq \tau,
\end{equation}
where the new constant $M=M(\tau)$ is some suitably large constant according to the value of $\tau>0$.
\end{remark}

From now on, we study the behavior of resolvent families on abstract interpolation scales associated with sectorial operators. For this end, let $A:\mathcal{D}(A)\subset X_0 \to X_0$ such that $-A$ is a sectorial operator and, for $j \in \mathbb{N}$, let $X_j$ be the space $\mathcal{D}(A^j)$ equipped with the norm $\|x\|_{X_j}:=\|(I-A)^jx\|_{X_0}$, where, for $j \geq 2$, $\mathcal{D}(A^j)$ is the space defined inductively by
$$\mathcal{D}(A^j)=\{x \in \mathcal{D}(A^{j-1}); Ax \in \mathcal{D}(A^{j-1})\}.$$
%
%
By an inductive argument, we can prove that $X_j$ is a complex Banach space, for all $j \in \mathbb{N}_0:=\mathbb{N}\cup\{0\}$, and 
$$...\,\hookrightarrow X_{j+1}\hookrightarrow X_j \hookrightarrow \,...\,\hookrightarrow X_2 \hookrightarrow X_1\hookrightarrow X_0$$
densely. Moreover, $-A$, regarded as an operator in $X_j$ with domain $\mathcal{D}(A^{j+1})$, is a sectorial operator as well. Denote by $\{\mathcal{F}_\theta\}_{\theta \in (0,1)}$ either the real or complex family of interpolation functors, see \cite{adams2003sobolev}. For each $k \in \mathbb{N}_0$, we define a {scale of intermediate spaces between} $(X_k,X_{k+1})$ by
$$X_{\alpha}:=\mathcal{F}_{\alpha-k}(X_k,X_{k+1}),\quad \alpha \in (k,k+1),$$
By the reiteration theorems for the real and complex methods, we have the following {reiteration property}:
\begin{equation}\label{reiteration property}
	\mathcal{F}_{\theta}(X_{k+\beta},X_{k+\gamma})=\mathcal{F}_{(1-\theta)\beta+\theta\gamma}(X_k,X_{k+1})=X_{k+(1-\theta)\beta+\theta\gamma},
\end{equation}
for $ k \in \mathbb{N}_0$, $0\leq \beta \leq\gamma \leq 1$ and $\theta \in (0,1).$
Then, if $j,k \in \mathbb{N}_0$, $\beta=(1-\theta)\beta_0+\theta\beta_1$ and $\gamma=(1-\theta)\gamma_0+\theta\gamma_1$, where $0\leq \beta_0\leq \beta_1\leq 1$, $0\leq \gamma_0\leq \gamma_1\leq 1,$ and  $\theta \in (0,1)$,
then,
$$X_{j+\beta}=\mathcal{F}_\theta(X_{j+\beta_0},X_{j+\beta_1}) \textrm{ and }X_{k+\gamma}=\mathcal{F}_\theta(X_{k+\gamma_0},X_{k+\gamma_1}),$$
with equivalent norms. Therefore, the following {interpolation inequality property} holds:
\begin{equation}\label{interpolation inequality property}
	\|T\|_{\mathcal{B}(X_{j+\beta},X_{k+\gamma})}\leq c\|T\|_{\mathcal{B}(X_{j+\beta_0},X_{k+\gamma_0})}^{1-\theta}\|T\|_{\mathcal{B}(X_{j+\beta_1},X_{k+\gamma_1})}^\theta,
\end{equation}
for all $T \in \mathcal{B}(X_{j+\beta_1},X_{k+\gamma_1})\cap\mathcal{B}(X_{j+\beta_1},X_{k+\gamma_1}).$
\begin{pr}\label{estimatives integers}
	If $i,j \in \mathbb{N}_0$ and $T:X_i \to X_j$ is a bounded linear operator such that
	$$Ty \in \mathcal{D}(A) \textrm{ and }ATy=TAy,\quad  \textrm{for all } y \in X_{i+1},$$
	then
	$$\|T\|_{\mathcal{B}(X_{i+k},X_{j+k})}\leq \|T\|_{\mathcal{B}(X_i,X_j)}, \quad i,j,k \in \mathbb{N}_0.$$
	\begin{proof}
		Indeed, given $x \in X_{i+k}$, we have
		\begin{align*}
			\|Tx\|_{X_{j+k}} &= \|(I-A)^kTx\|_{X_j} \\
			&=\|T(I-A)^kx\|_{X_j}\\
			&\leq \|T\|_{\mathcal{B}(X_i,X_j)}\|(I-A)^kx\|_{X_i}\\
			&=\|T\|_{\mathcal{B}(X_i,X_j)}\|x\|_{X_{i+k}}.
		\end{align*}
		Then,
		$$\|T\|_{\mathcal{B}(X_{i+k},X_{j+k})}\leq \|T\|_{\mathcal{B}(X_i,X_j)}.$$
	\end{proof}
\end{pr}
\begin{lema}\label{lemma for estimative in scale}
For $\alpha \in [0,1]$ there is a constant $c>0$ such that, for each non zero bounded linear operator $T:X_0\to X_1$ with $ATx=TAx$, $x \in X_1$, we have
	\begin{itemize}
		\item[(1)] $\|T\|_{\mathcal{B}(X_{\alpha},X_1)}\leq c\|T\|_{\mathcal{B}(X_0,X_1)}^{1-\alpha}
		\|T\|_{\mathcal{B}(X_0)}^{\alpha};$
		\item[(2)] $\|T\|_{\mathcal{B}(X_{\alpha},X_{1+\alpha})}\leq c\|T\|_{\mathcal{B}(X_0,X_1)};$
		\item[(3)] $\|T\|_{\mathcal{B}(X_1,X_{1+\alpha})}\leq c\|T\|_{\mathcal{B}(X_0)}^{1-\alpha}\|T\|_{\mathcal{B}(X_0,X_1)}^{\alpha}$.
	\end{itemize}
	\begin{proof}
For $\alpha =0$ or $\alpha =1$ all inequalities follow from Proposition \ref{estimatives integers}. Then, we can suppose $\alpha \in (0,1)$. In this case, the inequalities follow from the interpolation inequality property \eqref{interpolation inequality property} and Proposition \ref{estimatives integers}.
		Indeed, for (1), since $X_\alpha=\mathcal{F}_{\alpha}(X_0,X_1)$ and $X_1=\mathcal{F}_{\alpha}(X_1,X_1)$, we have
		$$\|T\|_{\mathcal{B}(X_{\alpha},X_1)}\leq c\|T\|_{\mathcal{B}(X_0,X_1)}^{1-\alpha}\|T\|_{\mathcal{B}(X_1,X_1)}^{\alpha} \leq c\|T\|_{\mathcal{B}(X_0,X_1)}^{1-\alpha}\|T\|_{\mathcal{B}(X_0)}^{\alpha}.$$
		Analogously, for (2),
		$$\|T\|_{\mathcal{B}(X_{\alpha},X_{1+\alpha})}\leq \|T\|_{\mathcal{B}(X_0,X_1)}^{1-\alpha}\|T\|_{\mathcal{B}(X_1,X_2)}^{\alpha}\leq c \|T\|_{\mathcal{B}(X_0,X_1)}.$$
		As for (3),
		$$\|T\|_{\mathcal{B}(X_1,X_{1+\alpha})}\leq c\|T\|_{\mathcal{B}(X_1,X_1)}^{1-\alpha}
		\|T\|_{\mathcal{B}(X_1,X_2)}^{\alpha}\leq c\|T\|_{\mathcal{B}(X_0)}^{1-\alpha}\|T\|_{\mathcal{B}(X_0,X_1)}^{\alpha}.
		$$
	\end{proof}
\end{lema}
\begin{lema}\label{lemma T from gamma to 1+theta}
	Let $0\leq \theta\leq\gamma\leq1$. There exists a constant $c>0$ such that, for each non zero bounded linear operator $T:X_0\to X_1$ satisfying $ATx=TAx$, $x \in X_1$, we have
	$$\|T\|_{\mathcal{B}(X_{\gamma},X_{1+\theta})}\leq c\|T\|_{\mathcal{B}(X_0,X_1)}^{1+\theta-\gamma}
	\|T\|_{\mathcal{B}(X_0)}^{\gamma-\theta}$$
	\begin{proof}
		The cases $\theta=0$, $\theta=\gamma$ and $\gamma=1$ correspond to Items (1), (2) and (3) of Lemma \ref{lemma for estimative in scale}, respectively. Then, we can suppose $0<\theta<\gamma<1$. Since $X_{1+\theta}=\mathcal{F}_{\theta/\gamma}(X_1,X_{1+\gamma})$ (reiteration property \eqref{reiteration property}) and $X_{\gamma}=\mathcal{F}_{\theta/\gamma}(X_\gamma,X_\gamma)$, by the interpolation inequality property \eqref{interpolation inequality property}, we have
		$$\|T\|_{\mathcal{B}(X_{\gamma},X_{1+\theta})}
		\leq c\|T\|_{\mathcal{B}(X_{\gamma},X_1)}^{1-\theta/\gamma}	\|T\|_{\mathcal{B}(X_{\gamma},X_{1+\gamma})}^{\theta/\gamma}.
		$$
		Then, by Lemma \ref{lemma for estimative in scale},
		\begin{align*}
			\|T\|_{\mathcal{B}(X_{\gamma},X_{1+\theta})}&\leq c\|T\|_{\mathcal{B}(X_{\gamma},X_1)}^{1-\theta/\gamma}	\|T\|_{\mathcal{B}(X_{\gamma},X_{1+\gamma})}^{\theta/\gamma}\\
			&\leq c\left(\|T\|_{\mathcal{B}(X_0,X_1)}^{1-\gamma}\|T\|_{\mathcal{B}(X_1)}^{\gamma}\right)^{1-\theta/\gamma}
			\left(\|T\|_{\mathcal{B}(X_0,X_1)}^{1-\gamma}\|T\|_{\mathcal{B}(X_1,X_2)}^{\gamma}\right)^{\theta/\gamma}\\
			&=  c\|T\|_{\mathcal{B}(X_0,X_1)}^{1-\gamma}
			\|T\|_{\mathcal{B}(X_1)}^{\gamma-\theta}
			\|T\|_{\mathcal{B}(X_1,X_2)}^{\theta} \\
			&\leq c\|T\|_{\mathcal{B}(X_0,X_1)}^{1+\theta-\gamma}
			\|T\|_{\mathcal{B}(X_0)}^{\gamma-\theta}.
		\end{align*}
	\end{proof}
\end{lema}

Suppose the pair $(A,g)$ satisfies the conditions of Theorem \ref{generation + regularity}. From the previous results we can prove the following smoothing effect of the resolvent family associated with the pair $(A,g)$.

\begin{theorem}\label{Theorem S from gamma to 1+theta}
	Let $0\leq \theta\leq\gamma\leq1$ and $\tau_0>0$. Then, there is $M=M(\tau_0,\theta,\gamma)>0$ such that
	$$\|S(t)\|_{\mathcal{B}(X_{\gamma},X_{1+\theta})}\leq Mt^{-\zeta_g(1+\theta-\gamma)}, \quad t \in (0,\tau_0].$$
	\begin{proof}
		From Lemma \ref{lemma T from gamma to 1+theta}, up to a constant factor,
		$$\|S(t)\|_{\mathcal{B}(X_{\gamma},X_{1+\theta})}\leq \|S(t)\|_{\mathcal{B}(X_0,X_1)}^{1+\theta-\gamma}
		\|S(t)\|_{\mathcal{B}(X_0)}^{\gamma-\theta}.$$
		Hence, from Remark \ref{aer}, we have
		\begin{align*}
			\|S(t)\|_{\mathcal{B}(X_{\gamma},X_{1+\theta})}\leq (Mt^{-\zeta_g})^{1+\theta-\gamma}		M^{\gamma-\theta}=Mt^{-\zeta_g(1+\theta-\gamma)}.
		\end{align*}
	\end{proof}
\end{theorem}

\begin{remark}
The above result is compatible with some previous results. Indeed, taking formally $\zeta_g=1$, we obtain the regularization of the analytic semigroup generated by $A$. On the other hand, considering $g$ such that
$$1\ast g(t)=\frac{t^\alpha}{\Gamma(\alpha+1)}, \quad t>0,$$
where $\alpha\in[0,1)$,  we can take $\zeta_g=1+\alpha$, see the Appendix for details. In this case the resolvent family associated with the pair $(A,g)$ satisfies
$$\|S(t)\|_{\mathcal{B}(X_{\gamma},X_{1+\theta})}\leq Mt^{-(1+\alpha)(1+\theta-\gamma)}, \quad t \in (0,\tau_0].$$
This particular case was firstly treated by de Andrade, Viana and Silva in \cite{andrade2021viana2021silva}.
\end{remark}

\begin{theorem}\label{Theorem delta S from gamma to 1+theta}
	Let $0\leq \theta\leq\gamma\leq1$ and $\tau_0>0$. Then, there is $M=M(\tau_0,\theta,\gamma)>0$ such that
	\begin{equation*}
		\|S(t_1)-S(t_0)\|_{\mathcal{B}(X_{\gamma},X_{1+\theta})} \leq M \ln(t_1/t_0)^{\gamma-\theta} t_0^{-\zeta_g(1+\theta-\gamma)}, 
	\end{equation*}
whenever $0<t_0<t_1\leq \tau_0$.	In particular, if $0\leq \theta<\gamma\leq 1$, then $$S(\cdot):(0,\infty)\to \mathcal{B}(X_\gamma,X_{1+\theta})$$ continuously.
	\begin{proof}
		By Remark \ref{aer}, there is some generic constant $M=M(\tau_0)$ such that
		\begin{align*}
			\|S(t_1)-S(t_0)\|_{\mathcal{B}(X_0,X_1)}&\leq \|S(t_1)\|_{\mathcal{B}(X_0,X_1)}+\|S(t_0)\|_{\mathcal{B}(X_0,X_1)}\nonumber\\
			&\leq M\left(t_1^{-\zeta_g}+t_0^{-\zeta_g}\right)\\
			&\leq  Mt_0^{-\zeta_g}
		\end{align*}
		and
		\begin{equation*}
			\left\|S(t_1)-S(t_0)\right\|_{\mathcal{B}(X_0)}\leq M \ln(t_1/t_0).
		\end{equation*}
		Then, by Lemma \ref{lemma T from gamma to 1+theta}, we have
		\begin{align*}
			\|S(t_1)-S(t_0)\|_{\mathcal{B}(X_{\gamma},X_{1+\theta})}&\leq  \|S(t_1)-S(t_0)\|_{\mathcal{B}(X_0,X_1)}^{1+\theta-\gamma}
			\|S(t_1)-S(t_0)\|_{\mathcal{B}(X_0)}^{\gamma-\theta}\\
			&\leq M\left(t_0^{-\zeta_g}\right)^{1+\theta-\gamma}\left(\ln\left(t_1/t_0\right)\right)^{\gamma-\theta} \\
			&=M\ln(t_1/t_0)^{\gamma-\theta}t_0^{-\zeta_g(1+\theta-\gamma)}.
		\end{align*}
	\end{proof}
\end{theorem}
\begin{pr}\label{proposition S s.continuous and bounded growth in compacts}
	Let $\beta \in (0,1)$. Then, for any compact $J \subset X_1$, we have $$\lim_{t\to0^+}\left(t^{\zeta_g\beta} \sup_{x \in J}\|S(t)x\|_{X_{1+\beta}}\right)=0.$$
	\begin{proof}
		Let $J \subset X_1$ be a compact set and $\delta>0$. Since $X_{1+\beta}\hookrightarrow X_1$ densely, the collection $\{B_\delta(y)\}_{y \in X_{1+\beta}}$, where
		$$B_\delta(y)=\left\lbrace z \in X_1; \, \|y-z\|_{X_1}<\delta\right\rbrace,$$
		is an open cover for $J$. Then, there exist $y_1, y_2, ..., y_n \in X_{1+\beta}$ such that $J \subset \cup_{k=1}^n B_\delta(y_k).$ Let $N = \max_{1\leq k \leq n}\{\|y_k\|_{X_{1+\beta}}\}$. Now, given $x \in J$, let $j \in \{1,2,...,n\}$ such that $x \in B_\delta(y_j)$. Then,
		\begin{align*}
			\|S(t)x\|_{X_{1+\beta}}&\leq  \|S(t)(x-y_j)\|_{X_{1+\beta}}+\|S(t)y_j\|_{X_{1+\beta}} \\
			&\leq  \|S(t)\|_{\mathcal{B}(X_1,X_{1+\beta})}\|x-y_j\|_{X_1}+\|S(t)\|_{\mathcal{B}(X_{1+\beta},X_{1+\beta})}\|y_j\|_{X_{1+\beta}}\\
			&\leq  \|S(t)\|_{\mathcal{B}(X_1,X_{1+\beta})}\delta+\|S(t)\|_{\mathcal{B}(X_{1+\beta},X_{1+\beta})}N.
		\end{align*}
		Hence, by Theorem \ref{Theorem S from gamma to 1+theta},
		\begin{align*}
			\|S(t)x\|_{X_{1+\beta}} \leq Mt^{-\zeta_g\beta}\delta + MN.
		\end{align*}
		Therefore,
		$$t^{\zeta_g\beta}\sup_{x \in J}\|S(t)x\|_{X_{1+\beta}}\leq M\left(\delta+t^{\zeta_g\beta}N\right)$$
		and
		$$ \limsup_{t \to 0^+}\left(t^{\zeta_g\beta}\sup_{x \in J}\|S(t)x\|_{X_{1+\beta}}\right)\leq M\delta.$$
		Finally, we pass the limit $\delta \to 0^+$.
	\end{proof}
\end{pr}

%
%


\section{Local well-posedness}
\sectionmark{Local well-posedness}

This section contains our main results on the well-posedness of problem \eqref{abstract equation in introduction}. To prove them, we suppose the pair $(A,g)$ satisfies the conditions of Theorem \ref{generation + regularity}. We divide our analysis according to the behavior of the nonlinear term $f$ present in \eqref{abstract equation in introduction} on the abstract interpolation scale $\{X_\alpha\}_{\alpha \geq 0}$  associated with $-A$.   However, before proceeding, we need to set the sense of solution we are looking for. To this end,  consider $\{S(t)\}_{t \geq 0}$ the analytic resolvent family  associated with the pair $(A,g)$.  Applying formally the Laplace transform to the resolvent equation \eqref{pqo}, we get
$$\widehat{S}(\lambda)x=\lambda^{-1}x+\lambda^{-1}\widehat{g}(\lambda)A\widehat{S}(\lambda)x$$
and so
$$	\widehat{S}(\lambda)x=\frac{1}{\widehat{g}(\lambda)}\left(\frac{\lambda}{\widehat{g}(\lambda)}-A\right)^{-1}x,$$
for $x\in \mathcal{D}(A)$. On the other hand, applying the Laplace transform to \eqref{abstract equation in introduction} we obtain 
$$\lambda\widehat{u}(\lambda)=u_0+\widehat{g}(\lambda)A\widehat{u}(\lambda)+\widehat{F}(\lambda).$$
where $\widehat{F}(\lambda)$ represent the Laplace transform of the nonlinear term $f(\cdot,u)$. Reorganizing, we have
$$\widehat{u}(\lambda)=\frac{1}{\widehat{g}(\lambda)}\left(\frac{\lambda}{\widehat{g}(\lambda)}-A\right)^{-1}u_0+\frac{1}{\widehat{g}(\lambda)}\left(\frac{\lambda}{\widehat{g}(\lambda)}-A\right)^{-1}\widehat{F}(\lambda).$$
Therefore, the inverse Laplace transform yields
$$u(t)=S(t)u_0+S\ast f(t,u(t)),\quad t\ge0.$$

The above argumentation and the previous related literature suggest the following definition.

\begin{de}[Mild solution]
Let $u_0 \in X_1$. A continuous function $u:[0,\tau]\to X_1$ is called {mild solution} for problem \eqref{abstract equation in introduction}	if it satisfies
	\begin{equation*}
		u(t)=S(t)u_0+\int_0^tS(t-s)f(s,u(s))ds,\quad t \in [0,\tau].
	\end{equation*}
\end{de}

\subsection{The subcritical case}
For the subcritical setting, we consider problem \eqref{abstract equation in introduction} with $f$ a Borel measurable $X_{\gamma_0}$-valued function on $(0,\infty)\times X_{1}$, for some $\gamma_0 \in (0,1)$. Furthermore, we suppose that there exist constants $c>0$ and $\rho>1$ such that, for all $x,y \in X_{1}$ and $t>0$, we have
	\begin{equation*}
		\|f(t,x)-f(t,y)\|_{X_{\gamma_0}}
		\leq c\|x-y\|_{X_{1}}\left(\|x\|^{\rho-1}_{X_{1}}+\|y\|^{\rho-1}_{X_{1}}+1\right)
	\end{equation*}
	and
	\begin{equation*}
		\|f(t,x)\|_{X_{\gamma_0}}\leq c\left(\|x\|^{\rho}_{X_{1}}+1\right).
	\end{equation*}

Before state our main result of this subsection, let us prove a technical lemma. Recall that, for $x,y>0$, the {Beta function} is defined by
$$\textbf{B}(x,y):=\int_0^1 s^{x-1}(1-s)^{y-1}ds.$$
Moreover, for $a \in (0,1]$, the {(lower) incomplete Beta function} is given by
$$\textbf{B}_a(x,y):=\int_0^as^{x-1}(1-s)^{y-1}ds=\int_{1-a}^1(1-s)^{x-1}s^{y-1}ds.$$

\begin{lema}\label{Lemma I kappa goes to zero as kappa goes to 1+}
	Fixed $a>0$ and $b,c<1$, the integral
	\begin{equation}\label{drr}
		I_\kappa=\int_0^1 \left(\ln\left(\frac{\kappa-s}{1-s}\right)\right)^{a}(1-s)^{-b}s^{-c}\, d s
	\end{equation}
	converges whenever $\kappa>1$. Moreover, $I_\kappa \to 0$ as $\kappa \to 1^+$.
	\begin{proof}
		The veracity of these statements becomes more evident as we rewrite
		$$I_\kappa = \int_0^1 \left(h_\kappa(s)\right)^{a}(1-s)^{-(b+\delta a)}s^{-c}ds,$$
		where
		$$h_\kappa(s)=\ln\left(\frac{\kappa-s}{1-s}\right)(1-s)^{\delta}, \quad s \in (0,1),$$
		and $\delta>0$ is chosen so that $b+\delta a<1$. Now, we note that $h_\kappa$ is bounded in $(0,1)$, since $\lim_{s\to0^+}h_\kappa(s)=\ln(\kappa)$ and $\lim_{s \to 1^-}h_\kappa(s)=0$. Then
		\begin{align*}
			I_\kappa&\leq \left(\sup_{0<s<1}h_\kappa(s)\right)^a\int_0^1 (1-s)^{-(b+\delta a)}s^{-c}ds\\
			&= \left(\sup_{0<s<1}h_\kappa(s)\right)^a\textbf{B}(1-(b+\delta a),1-c)<\infty.
		\end{align*}
		Moreover, since the integrating function in \eqref{drr} decreases pointwise to $0$ as $\kappa \to 1^+$, it follows from the monotone convergence theorem that $\lim_{\kappa \to 1^+}I_\kappa=0$.
	\end{proof}
\end{lema}

The following is our main result in the subcritical context .
\begin{theorem}\label{thmsubcritico}
Suppose the pair $(A,g)$ satisfies the conditions of Theorem \ref{generation + regularity}.	If $1-\frac{1}{\zeta_g}<\gamma_0<1$,
then, given $x_0 \in X_1$ we can consider $\tau>0$ and $r>0$ such that for any $u_0 \in B_{X_1}[x_0,r]=\{x \in X_1\big/ \,\|x-x_0\|_{X_1}\leq r\}$ problem \eqref{abstract equation in introduction} has a unique mild solution $u(\cdot\ ; u_0) \in C\left([0,\tau];X_1\right)$. Moreover, the following statements hold: 
	\begin{itemize}
	\item[{(A)}] for all $\theta \in (0, \gamma_0-1+1/\zeta_g)$, we have
	$$u(\cdot,u_0) \in C((0,\tau];X_{1+\theta})$$
	and, if $J \subset X_1$ is compact, then
	$$\lim_{t\to0^+}t^{\zeta_g\theta}\sup_{u_0 \in J}\|u(t;u_0)\|_{X_{1+\theta}}=0.$$
	\item[{(B)}] For each $\theta \in [0, \gamma_0-1+1/\zeta_g)$, there exists a constant $L>0$ such that
	$$t^{\zeta_g\theta}\|u(t;u_0)-u(t;u_1)\|_{X_{1+\theta}}\leq L \|u_0-u_1\|_{X_1}, \quad t\in (0,\tau].$$
	whenever $ u_0,u_1 \in B_{X_1}[x_0,r]$.
	%
\end{itemize}
\end{theorem}
\begin{proof}
Initially, note that $1-\frac{1}{\zeta_g}<\gamma_0$ is equivalent to $1-\zeta_g(1-\gamma_0)>0$. Hence, consider $0<\mu\le 1$ and $\tau>0$ such that 
$$\|S(t)x_0-x_0\|_{X_1}\le \frac{\mu}{4} \quad \mbox{and}\quad MRc{\bf B}(1,1-\zeta_g(1-\gamma_0))t^{1-\zeta_g(1-\gamma_0)}\le  \frac{\mu}{4}$$
for any $t\in[0,\tau]$, where
$$R=\max\left\{(2(\|x_0\|_{X_1}+\mu)^{\rho-1}+1, (\|x_0\|+1)^{\rho}_{X_{1}}+1 \right\}.$$ 
Fix $r=\frac{\mu}{2M}$ and Consider the closed ball $\mathcal{K}$ given by
	$$\mathcal{K}=\left\lbrace \phi \in C\left([0,\tau];X_{1}\right)\big/\, \sup_{0\leq t\leq \tau}\|\phi(t)-x_0\|_{X_{1}}\leq \mu\right\rbrace.$$
Define the map $T$ on ${\mathcal{K}}$ by
	$$(T \phi)(t)=S(t)u_0+\int_0^t S(t-s)f(\phi (s))ds, \quad t \in [0,\tau].$$

Let us prove that $\mathcal{K}$ is $T$-invariant. Given $\phi\in\mathcal{K}$, it follows from Theorem \ref{Theorem delta S from gamma to 1+theta} that $S(\cdot)u_0$ is a continuous function. As for the map
$$t \mapsto \int_0^t S(t-s)f(\phi(s))ds,$$ 
note that for fixed $t_1,t_0 \in (0,\tau]$, with $t_1<t_0$, we have
\begin{align}\label{2s poi0}
	&\left\|\int_{0}^{t_1}S(t_1-s)f(\phi(s))ds-\int_{0}^{t_0}S(t_0-s)f(\phi (s))ds\right\|_{X_1} \nonumber\\
	\leq&\,\left\|\int_{0}^{t_1}\left(S(t_1-s)-S(t_0-s)\right)f(\phi (s))ds\right\|_{X_1}
	+\left\|\int_{t_1}^{t_0}S(t_0-s)f(\phi (s))ds\right\|_{X_1}.
\end{align}
Applying Theorem \ref{Theorem delta S from gamma to 1+theta} to the first term of this sum we get
\begin{align*}
	&\left\|\int_{0}^{t_1}\left(S(t_1-s)-S(t_0-s)\right)f(\phi (s))ds\right\|_{X_1}\\ 
	&\leq \int_{0}^{t_1}\|\left(S(t_1-s)-S(t_0-s)\right)f(\phi (s))\|_{X_{1}}ds\\
	&\leq Mc \int_{0}^{t_1}\ln\left(\frac{t_0-s}{t_1-s}\right)^{\gamma_0}(t_1-s)^{-\zeta_g(1-\gamma_0)}\left(\|\phi (s)\|^{\rho}_{X_1}+1\right)ds\\
	&\leq Mc\big((\|x_0\|+\mu)^{\rho}_{X_{1}}+1\big)\int_{0}^{t_1}\ln\left(\frac{t_0-s}{t_1-s}\right)^{\gamma_0}(t_1-s)^{-\zeta_g(1-\gamma_0)}ds\\
	&\leq Mc\big((\|x_0\|+\mu)^{\rho}_{X_{1}}+1\big) t_1^{-\zeta_g(1-\gamma_0)}\int_{0}^{1}\ln\left(\frac{t_0/t_1-s}{1-s}\right)^{\gamma_0}(1-s)^{-\zeta_g(1-\gamma_0)}ds,
\end{align*}
which, by Lemma \ref{Lemma I kappa goes to zero as kappa goes to 1+}, converges to $0$ as $|t_1-t_0|\to 0$. In the same way, the second term in \eqref{2s poi0} satisfies
\begin{align*}
	\left\|\int_{t_1}^{t_0}S(t_0-s)f(\phi (s))ds\right\|_{X_1}	&\leq
	Mc\big((\|x_0\|+\mu)^{\rho}_{X_{1}}+1\big)\int_{t_1}^{t_0}(t_0-s)^{-\zeta_g(1-{\gamma_0})}ds\\
	&=\,
	Mc\big((\|x_0\|+\mu)^{\rho}_{X_{1}}+1\big)\frac{(t_0-t_1)^{1-\zeta_g(1-\gamma_0)}}{1-\zeta_g(1-\gamma_0)},
\end{align*}
which also goes to $0$ as $|t_1- t_0|\to 0$; consequently\footnote{The case $t_0<t_1$ is analogous.}, $T\phi\in C\left([0,\tau];X_{1}\right)$. Also, if $\phi\in \mathcal{K}$, then

	\begin{align*}
		\|(T\phi)(t)-x_0\|_{X_1}  &\leq \|S(t)u_0-S(t)x_0\|_{X_1}+\|S(t)x_0-x_0\|_{X_1}+
		\int_0^t\|S(t-s)f(\phi(s))\|_{X_{1}}ds \\
		&\leq \frac{\mu}{2}+\frac{\mu}{4}+M\int_0^t(t-s)^{-\zeta_g(1-\gamma_0)}\|f(\phi(s))\|_{X_{\gamma_0}}ds\\
		&\leq \frac{3\mu}{4}+Mc\int_0^t(t-s)^{-\zeta_g(1-\gamma_0)}\big(\|\phi(s)\|^{\rho}_{X_{1}}+1\big)ds\\
		&\leq \frac{3\mu}{4}+Mc\big((\|x_0\|+\mu)^{\rho}_{X_{1}}+1\big)\int_0^t(t-s)^{-\zeta_g(1-\gamma_0)}ds\\
		&\leq \frac{3\mu}{4}+Mc\big((\|x_0\|+\mu)^{\rho}_{X_{1}}+1\big){\bf B}(1,1-\zeta_g(1-\gamma_0))t^{1-\zeta_g(1-\gamma_0)}\\
		& \le  \frac{3\mu}{4} + \frac{\mu}{4}=\mu.
	\end{align*}

Furthermore, $T$ is a contraction on $\mathcal{K}$. Indeed, for $\phi,\psi \in \mathcal{K}$, we have
\begin{align*}
	&	\|(T\phi)(t)-(T\psi)(t)\|_{X_1}  \leq	   \int_0^t\left\|S(t-s)\big(f(\phi(s))-f(\psi(s))\big)\right\|_{X_{1}}ds\\		
	&\leq Mc\int_0^t (t-s)^{-\zeta_g(1-\gamma_0)}\left(\|\phi(s)\|^{\rho-1}_{X_{1}}+\|\psi(s)\|^{\rho-1}_{X_{1}}+1\right)\|\phi(s)-\psi(s)\|_{X_1}ds \\
	&\leq \left(Mc\big(2(\|x_0\|_{X_1}+\mu)^{\rho-1}+1\big)\int_0^t (t-s)^{-\zeta_g(1-\gamma_0)}ds\right)\sup_{0\leq t\leq \tau}\|\phi(t)-\psi(t)\|_{X_1}\\
	&\leq Mc\big(2(\|x_0\|_{X_1}+\mu)^{\rho-1}+1\big){\bf B}(1,1-\zeta_g(1-\gamma_0))t^{1-\zeta_g(1-\gamma_0)}\sup_{0\leq t\leq \tau}\|\phi(t)-\psi(t)\|_{X_1}\\
	&\le \frac{1}{4}\sup_{0\leq t\leq \tau}\|\phi(t)-\psi(t)\|_{X_1}.
\end{align*}
The Banach fixed point theorem ensures the existence of a unique $\phi\in\mathcal{K}$ such that 
$$\phi(t)=(T \phi)(t)=S(t)u_0+\int_0^t S(t-s)f(\phi (s))ds, \quad t \in [0,\tau],$$
that is, the  function $u(t;u_0):=\phi(t)$ is a mild solution of \eqref{abstract equation in introduction} in $[0,\tau]$. If $v\in C\left([0,\tau];X_{1}\right)$ is another mild solution of \eqref{abstract equation in introduction}, then
\begin{align*}
	&	\|u(t;u_0)-v(t)\|_{X_1}  \leq \int_0^t\left\|S(t-s)\big(f(u(s;u_0))-f(v(s))\big)\right\|_{X_{1}}ds\\		
	&\leq Mc\int_0^t (t-s)^{-\zeta_g(1-\gamma_0)}\left(\|u(s;u_0)\|^{\rho-1}_{X_{1}}+\|v(s)\|^{\rho-1}_{X_{1}}+1\right)\|u(s;u_0)-v(s)\|_{X_1}ds \\
	&\leq b\int_0^t (t-s)^{-\zeta_g(1-\gamma_0)}\|u(s;u_0)-v(s)\|_{X_1}ds,
\end{align*}
where 
$$b:=Mc\sup_{0\le s\le \tau}\left(\|u(s;u_0)\|^{\rho-1}_{X_{1}}+\|v(s)\|^{\rho-1}_{X_{1}}+1\right)<\infty.$$
Applying the singular Gronwall inequality, see \cite[Lemma 7.1.1, page 188]{henry2006geometric}, we conclude that $u(t;u_0)=v(t)$, for all $t\in [0,\tau]$.

In order to conclude the proof, we consider the remaining items.

\noindent{\bf Proof of (A):} Consider $\theta \in (0, \gamma_0-1+1/\zeta_g)$. For all $t\in(0,\tau]$ we have
\begin{align*}
	&	\|u(t;u_0)\|_{X_{1+\theta}}  \leq \|S(t)u_0\|_{X_{1+\theta}}+\int_0^t\left\|S(t-s)f(u(s;u_0))\right\|_{X_{1+\theta}}ds\\		
	&\leq Mt^{-\zeta_g \theta}\|u_0\|_{X_1}+Mc\int_0^t (t-s)^{-\zeta_g(1+\theta-\gamma_0)}\left(\|u(s;u_0)\|^{\rho}_{X_{1}}+1\right)ds \\
	&\leq Mt^{-\zeta_g \theta}\|u_0\|_{X_1}+Mc\big((\|u_0\|_{X_1}+\mu)^\rho+1\big)t^{1-\zeta_g(1+\theta-\gamma_0)}{\bf B}(1,1-\zeta_g(1+\theta-\gamma_0)).
\end{align*}
Therefore, $u(\cdot\ ; u_0 ):(0,\tau]\to X_{1+\theta}$ is well defined. A similar argument to that used before proves that $u(\cdot\ ; u_0 )\in C((0,\tau];X_{1+\theta})$. Moreover, from the above estimate, if $t>0$ we have that
\begin{align*}
t^{\zeta_g \theta}\|u(t;u_0)\|_{X_{1+\theta}}  &\leq t^{\zeta_g \theta}\|S(t)u_0\|_{X_{1+\theta}}\\ &+Mc\big((\|u_0\|_{X_1}+\mu)^\rho+1\big)t^{1-\zeta_g(1-\gamma_0)}{\bf B}(1,1-\zeta_g(1+\theta-\gamma_0)),
\end{align*}
which implies that the right side of the above inequality goes to $0$ as $t\to 0^+$; we conclude the proof of (A) using Proposition \ref{proposition S s.continuous and bounded growth in compacts}.

\noindent{\bf Proof of (B):} If $u_0,u_1\in X_1$ then
\begin{align*}
		t^{\zeta_g \theta}\|u(t;u_0)-u(t,u_1)\|_{X_{1+\theta}} & \leq 
	 M\|u_0-u_1\|_{X_1}\\
	 & + \Gamma_{\theta}(t)\sup_{0\le t \leq  \tau}\|u(t;u_0)-u(t;u_1)\|_{X_{1}},
\end{align*}
where
$$ \Gamma_{\theta}(t)=Mct^{1-\zeta_g(1-\gamma_0)}{\bf B}(1,1-\zeta_g(1+\theta-\gamma_0))\left(\sup_{0\le t \leq  \tau}\|u(t;u_0)\|_{X_{1}}+\sup_{0\le t \leq  \tau}\|u(t;u_1)\|_{X_{1}}+1\right),$$
for all $t\in[0,\tau]$. Taking $\theta=0$ we have
$$\|u(t;u_0)-u(t,u_1)\|_{X_{1}}  \leq 
M\|u_0-u_1\|_{X_1} + \frac{1}{4}\sup_{0\le t \leq  \tau}\|u(t;u_0)-u(t;u_1))\|_{X_{1}},$$
and then
$$\sup_{0\le t \leq  \tau}\|u(t;u_0)-u(t;u_1))\|_{X_{1}}\le \frac{4}{3}\ M\|u_0-u_1\|_{X_1}.$$
Consequently, for $0\le \theta \le \theta_0< \gamma_0-1+1/\zeta_g$ it follows that
$$
t^{\zeta_g \theta}\|u(t;u_0)-u(t,u_1)\|_{X_{1+\theta}}  \leq 
	M\|u_0-u_1\|_{X_1} +\frac{4}{3}\ M \Gamma_{\theta}(t) \|u_0-u_1\|_{X_1},
$$
that is 
$$t^{\zeta_g \theta}\|u(t;u_0)-u(t,u_1)\|_{X_{1+\theta}}  \leq 
L\|u_0-u_1\|_{X_1},$$
with 
$$L=M\left(1+\frac{4}{3}\sup\left\{\Gamma_{\theta}(t)\big/ 0\le t\le \tau,\ 0\le \theta\le \theta_0\right\}\right),$$
which concludes the proof.
\end{proof} 

\subsection{The critical case}\label{section critical case}

In this subsection, we consider problem \eqref{abstract equation in introduction} in the critical framework. As we discussed in the introduction of this paper, assuming only that the function $f$ present in \eqref{abstract equation in introduction} is a locally Lipschitz function $f(t,\cdot):X_{1} \to X_{\gamma}$, for some $\gamma\le 1-\frac{1}{\zeta_g}$, it seems to be impossible to ensure that \eqref{abstract equation in introduction} is well-posed in general. To overcome this situation, we use the notion of $\varep$-regular map introduced by Arrieta and Carvalho \cite{arrieta2000abstract}.   

Particularly, we assume that there exist $\varep, \gamma(\varep) \in (0,1)$ such that $f(t,\cdot):X_{1+\varep}\to X_{\gamma(\varep)}$ is well-defined for each $t>0$, and the function $f(\cdot,x):(0,\infty)\to X_{\gamma(\varep)}$ is measurable for each $x \in X_{1+\varep}$. Furthermore, there exist constants $c>0$, $\rho>1$ and a non-decreasing function $\nu
	(\cdot)$ with $\lim_{t\to0^+}\nu(t)=0$, such that, for all $x,y \in X_{1+\varep}$ and $t>0$, we have
	\begin{align*}
		\|f(t,x)-f(t,y)\|_{X_{\gamma(\varep)}} \leq c\|x-y\|_{X_{1+\varep}}\left(\|x\|^{\rho-1}_{X_{1+\varep}}+\|y\|^{\rho-1}_{X_{1+\varep}}+\nu(t)t^{-1+\zeta_g(1-\gamma(\varep)+\varep)}\right)
	\end{align*}
	and
	\begin{equation*}
		\|f(t,x)\|_{X_{\gamma(\varep)}}\leq c\left(\|x\|^{\rho}_{X_{1+\varep}}+\nu(t)t^{-1+\zeta_g(1-\gamma(\varep))}\right).
	\end{equation*}

\begin{remark}
Although we assume that \( u_0 \in X_1 \), we make no assumptions about the behavior of \( f \) on \( X_1 \), as dictated by the critical setting.
\end{remark}

Before proceeding, let us define the type of solution we are looking for. 

\begin{de}[$\varep$-regular mild solution]
	A mild solution $u \in  C\left([0,\tau];X_1\right)$  to \eqref{abstract equation in introduction} is said to be \emph{$\varep$-regular} on $[0,\tau]$ if $u \in C\left((0,\tau];X_{1+\varep}\right)$.
\end{de} 

The following lemmas will be useful to prove our main result.

\begin{lema}\label{lemma critical case}
	Let $\tau_0,\mu >0$ be fixed constants and consider the set
	$$\mathcal{K}(\tau_0,\mu)=\left\lbrace \phi \in C\left((0,\tau_0];X_{1+\varep}\right); \, \|\phi\|_{\mathcal{K}(\tau_0,\mu)}:=\sup_{0<t\leq \tau_0}t^{\zeta_g\varep}\|\phi(t)\|_{X_{1+\varep}}\leq \mu\right\rbrace.$$
	Given a sub-interval $(a,b]\subset (0,\tau_0]$ and functions $\phi, \psi \in \mathcal{K}(\tau_0,\mu)$, for each $\theta \in [0,\gamma(\varep)-1+1/\zeta_g)$ and all $t \in (a,b]$, we have
	\begin{align}\label{inequality 1 lemma critical case}
		&t^{\zeta_g\theta}\left\|\int_{a}^{t}S(t-s)f(s,\phi (s))ds\right\|_{X_{1+\theta}}\nonumber
		\\
		\leq&\,
		m(\theta,a,b)\left(\left(\sup_{0<s\leq t}s^{\zeta_g\varep}\|\phi(s)\|_{X_{1+\varep}}\right)^\rho t^{1-\zeta_g(1-\gamma(\varep)+\rho\varep)} +\nu(t)\right)
	\end{align}
	and	
	\begin{align}\label{inequality 2 lemma critical case}
		&t^{\zeta_g\theta}\left\|\int_{a}^t S(t-s)\left(f(s,\phi (s))-f(s,\psi(s))\right)ds\right\|_{X_{1+\theta}}\nonumber\\
		\leq&\, 
		m(\theta,a,b)\left(2\mu^{\rho-1}t^{1-\zeta_g(1-{\gamma(\varep)}+\rho\varep)}+\nu(\tau_0)\right) \sup_{0<s\leq t}s^{\zeta_g\varep}\|\phi(s)-\psi(s)\|_{X_{1+\varep}},
	\end{align}
	where $m(\theta,a,b)$ is the biggest of the constants
	$$Mc \textbf{B}_{1-a/b}(1-\zeta_g\varep\rho,1-\zeta_g(1+\theta-\gamma(\varep)))$$
	and	$$Mc\textbf{B}_{1-a/b}\left(\zeta_g(1-\gamma(\varep)),1-\zeta_g(1+\theta-\gamma(\varep))\right).$$
	\begin{proof}First, let us prove \eqref{inequality 1 lemma critical case}. For the sake of simplicity, we denote 
		$$\lambda(t)=\sup_{0<s\leq t}s^{\zeta_g\varep}\|\phi(s)\|_{X_{1+\varep}}.$$
		We note that
		\begin{align*}\label{jui}
			&\left\|\int_a^{t}S(t-s)f(s,\phi (s))ds\right\|_{X_{1+\theta}}\leq \int_a^t \|S(t-s)\|_{\mathcal{B}(X_{\gamma(\varep)};X_{1+\theta})}\|f(s,\phi (s))\|_{X_{\gamma(\varep)}}ds\nonumber\\
			\leq&\,
			Mc\int_a^{t}(t-s)^{-\zeta_g(1+\theta-\gamma(\varep))}\left(\|\phi (s)\|_{X_{1+\varep}}^\rho+\nu(s)s^{-1+\zeta_g(1-\gamma(\varep))}\right)ds \nonumber\\
			\leq&\, 
			Mc \int_a^t(t-s)^{-\zeta_g(1+\theta-{\gamma(\varep)})}\left(s^{-\zeta_g\varep\rho}\lambda(t)^\rho+\nu(t)s^{-1+\zeta_g(1-\gamma(\varep))}\right)ds\nonumber\\
			=&\,
			Mc\left(\lambda(t)^\rho\int_a^{t}(t-s)^{-\zeta_g(1+\theta-{\gamma(\varep)})}s^{-\zeta_g\varep\rho}ds +\nu(t)\int_a^t(t-s)^{-\zeta_g(1+\theta-{\gamma(\varep)})}s^{-1+\zeta_g(1-\gamma(\varep))}ds\right).
		\end{align*}
		Performing the substitution $s=st$ in the last two integrals above, we have 
		\begin{align*}
			\left\|\int_a^{t}S(t-s)f(s,\phi (s))ds\right\|_{X_{1+\theta}}
			&\leq
			Mc\lambda(t)^\rho t^{1-\zeta_g(1+\theta-{\gamma(\varep)})-\zeta_g\varep\rho}\int_{a/t}^{1}(1-s)^{-\zeta_g(1+\theta-{\gamma(\varep)})}s^{-\zeta_g\varep\rho}ds \nonumber\\
			&+Mc\nu(t)t^{-\zeta_g\theta}\int_{a/t}^1(1-s)^{-\zeta_g(1+\theta-{\gamma(\varep)})}s^{-1+\zeta_g(1-\gamma(\varep))}ds\nonumber\\
			&\leq\,
			Mc\lambda(t)^\rho t^{1-\zeta_g(1+\theta-{\gamma(\varep)}+\varep\rho)}\int_{a/b}^{1}(1-s)^{-\zeta_g(1+\theta-{\gamma(\varep)})}s^{-\zeta_g\varep\rho}ds \nonumber\\
			&+Mc\nu(t)t^{-\zeta_g\theta}\int_{a/b}^1(1-s)^{-\zeta_g(1+\theta-{\gamma(\varep)})}s^{-1+\zeta_g(1-\gamma(\varep))}ds\nonumber\\
			&=\,
			Mc\lambda(t)^\rho t^{1-\zeta_g(1+\theta-{\gamma(\varep)}+\varep\rho)}\textbf{B}_{1-a/b}(1-\zeta_g\varep\rho,1-\zeta_g(1+\theta-{\gamma(\varep)}))\nonumber\\
			&+
			Mc\nu(t)t^{-\zeta_g\theta}\textbf{B}_{1-a/b}(\zeta_g(1-\gamma(\varep)),1-\zeta_g(1+\theta-{\gamma(\varep)}))\nonumber\\
			&\leq\,
			m(\theta,a,b)\left(\lambda(t)^\rho t^{1-\zeta_g(1+\theta-{\gamma(\varep)}+\varep\rho)}+\nu(t)t^{-\zeta_g\theta}\right).
		\end{align*}
		As for \eqref{inequality 2 lemma critical case}, denote $$Q(t)=\sup_{0<s\leq t}s^{\zeta_g\varep}\|\phi(s)-\psi(s)\|_{X_{1+\varep}}.$$
		We observe that
		
		\begin{align*}
			&\left\|\int_a^t S(t-s)\left(f(s,\phi (s))-f(s,\psi(s))\right)ds\right\|_{X_{1+\theta}} \\
			&\leq\,
			\int_a^{t}\|S(t-s)\|_{\mathcal{B}(X_{\gamma(\varep)},X_{1+\theta})}\|f(s,\phi(s))-f(s,\psi(s))\|_{X_{{\gamma(\varep)}}}ds\\
			&\leq\,
			Mc\int_a^{t}(t-s)^{-\zeta_g(1+\theta-{\gamma(\varep)})}\|\phi (s)-\psi(s)\|_{X_{1+\varep}}\\
			&\qquad \times \left(\|\phi (s)\|_{X_{1+\varep}}^{\rho-1}+\|\psi(s)\|_{X_{1+\varep}}^{\rho-1}+\nu(s)s^{-1+\zeta_g(1-\gamma(\varep)+\varep)}\right)ds \\
			&\leq\,
			Mc\int_a^t (t-s)^{-\zeta_g(1+\theta-{\gamma(\varep)})}s^{-\zeta_g\varep} Q(t)\nonumber\\
			&\qquad
			\times\left(2s^{-\zeta_g\varep(\rho-1)}\mu^{\rho-1}+\nu(\tau_0)s^{-1+\zeta_g(1-\gamma(\varep)+\varep)}\right)ds\\
			&=\,
			2Mc Q(t)\mu^{\rho-1}\int_a^t (t-s)^{-\zeta_g(1+\theta-{\gamma(\varep)})}s^{-\zeta_g\varep\rho}ds\\
			&+
			Mc Q(t)\nu(\tau_0)\int_a^t (t-s)^{-\zeta_g(1+\theta-{\gamma(\varep)})}s^{-1+\zeta_g(1-\gamma(\varep))}ds.
		\end{align*}
		Then,
		
		\begin{align*}
			&\left\|\int_a^t S(t-s)\left(f(s,\phi (s))-f(s,\psi(s))\right)ds\right\|_{X_{1+\theta}} \\
			\leq&\,
			2Mc Q(t)\mu^{\rho-1}t^{1-\zeta_g(1+\theta-{\gamma(\varep)}+\rho\varep)}\int_{a/t}^1 (1-s)^{-\zeta_g(1+\theta-{\gamma(\varep)})}s^{-\zeta_g\varep\rho}ds\\
			&+
			Mc Q(t)\nu(\tau_0)t^{-\zeta_g\theta}\int_{a/t}^1 (1-s)^{-\zeta_g(1+\theta-{\gamma(\varep)})}s^{-1+\zeta_g(1-\gamma(\varep))}ds\\
			\leq&\,
			2Mc Q(t)\mu^{\rho-1}t^{1-\zeta_g(1+\theta-{\gamma(\varep)}+\rho\varep)}\textbf{B}_{1-a/b}(1-\zeta_g\varep\rho,1-\zeta_g(1+\theta-{\gamma(\varep)}))\\
			&+
			Mc Q(t)\nu(\tau_0)t^{-\zeta_g\theta}\textbf{B}_{1-a/b}(\zeta_g(1-\gamma(\varep)),1-\zeta_g(1+\theta-{\gamma(\varep)}))\\
			\leq&\,
			m(\theta,a,b)\left(2\mu^{\rho-1}t^{1-\zeta_g(1+\theta-{\gamma(\varep)}+\rho\varep)}+\nu(\tau_0)t^{-\zeta_g\theta}\right) Q(t).
		\end{align*}
	\end{proof}
\end{lema}

The next lemma is a direct consequence of \cite[Lemma 7.1.2]{henry2006geometric}.

\begin{lema}[Singular Grönwall's Inequality]\label{Grönwall's inequality}
	Let $b \geq 0$, $\alpha<1$, $0\leq \tau<T<\infty$, and $\phi \in L^1_{\mathrm{loc}}([\tau,T))$. If
	$$|\phi(t)|\leq b\int_\tau^t(t-s)^{-\alpha}|\phi(s)|ds, \quad t \in [\tau,T),$$
	then $\phi\equiv 0$ a.e. on $[\tau,T)$.
\end{lema}


The following is our main result in the critical case.

\begin{theorem}\label{theorem critic case well-posedness}
Suppose the pair $(A,g)$ satisfies the conditions of Theorem \ref{generation + regularity}. If $\gamma(\varep)\ge 1-\frac{1}{\zeta_g}+\rho\varep,$	then, for all $x_0 \in X_1$, there exist $r=r(x_0)>0$ and $\tau_0=\tau_0(x_0)>0$ such that problem \eqref{abstract equation in introduction} has an $\varep$-regular mild solution $u(\cdot\,;u_0)$ defined on $[0,\tau_0]$ for all $u_0 \in B_{X_1}[x_0,r]=\{x \in X_1;\,\|x-x_0\|_{X_1}\leq r\}$. Moreover, the following statements hold:
	\begin{itemize}
		\item[{(A)}] For all $\theta \in (0, \gamma(\varep)-1+1/\zeta_g)$ (in particular, for $\theta=\varep$), we have
		$$u(\cdot\,;u_0) \in C\left((0,\tau_0];X_{1+\theta}\right), \quad u_0 \in B_{X_1}[x_0,r],$$
		and if $J \subset B_{X_1}[x_0,r]$ is compact then
		$$\lim_{t\to0^+}t^{\zeta_g\theta}\sup_{u_0 \in J}\|u(t;u_0)\|_{X_{1+\theta}}=0.$$
		\item[{(B)}] For each $\theta \in [0, \gamma(\varep)-1+1/\zeta_g)$, there exists a constant $L=L(\theta,x_0)$ such that
		$$t^{\zeta_g\theta}\|u(t;u_0)-u(t;v_0)\|_{X_{1+\theta}}\leq L \|u_0-v_0\|_{X_1}, \quad t\in (0,\tau_0],\, u_0,v_0 \in B_{X_1}[x_0,r].$$
		\item[{(C)}] If $v:[0,\tau_1]\to X_1$ is an $\varep$-regular mild solution on $[0,\tau_1]$ for the problem \eqref{abstract equation in introduction} starting in $u_0 \in B_{X_1}[x_0,r]$ and $v$ also satisfies
		$$\lim_{t \to 0^+}t^{\zeta_g\varep}\|v(t)\|_{X_{1+\varep}}=0,$$
		then $v(t) = u(t;u_0)$ for all $t \in [0,\min\{\tau_0,\tau_1\}]$.
		\item[{(D)}] If $\zeta_g<(1-\gamma(\varep)+\rho\varep)^{-1}$, then $r>0$ can be taken arbitrarily large, that is, the time of existence is uniform on bounded sets of $X_1$.
	\end{itemize}
\end{theorem}

\begin{proof}
As in Lemma \ref{lemma critical case}, for each $\theta \in [0,\gamma(\varep)-1+1/\zeta_g)$ and $0\leq a <b<\infty$, we denote by $m(\theta,a,b)$ the biggest of the constants
$$Mc \textbf{B}_{1-a/b}(1-\zeta_g\varep\rho,1-\zeta_g(1+\theta-\gamma(\varep)))$$
and	$$Mc\textbf{B}_{1-a/b}\left(\zeta_g(1-\gamma(\varep)),1-\zeta_g(1+\theta-\gamma(\varep))\right).$$

Fixed $x_0 \in X_1$, let $\tau_0=\tau_0(x_0) \in (0,1]$ arbitrarily small and $\mu=\mu(x_0)>0$ such that
\begin{equation}\label{gyg}
	m(\varep,0,\tau_0)\mu^{\rho-1}t^{1-\zeta_g(1-\gamma(\varep)+\rho\varep)}\leq \frac{1}{4}, \quad 0<t\leq \tau_0,
\end{equation}
\begin{equation}\label{ggg}
	t^{\zeta_g\varep}\|S(t)x_0\|_{X_{1+\varep}}\leq \frac{\mu}{4}, \quad 0<t\leq \tau_0,
\end{equation}
and
\begin{equation}\label{gvg}
	m(\varep,0,\tau_0)\nu(\tau_0)\leq \min\{1/4,\mu/4\}.
\end{equation}
Then, let us define
\begin{equation}\label{g g}
	r(x_0)=\frac{\mu}{8M}.
\end{equation} 
We note that a such pair $(\tau_0,\mu)$ exists, since
\begin{itemize}
	\item $\rho>1$ and $\lim_{\tau_0\to 0^+}\nu(\tau_0)=0,$ by our hypotheses on $f$;
	\item $m(\varep,0,\tau_0)$ does not depends upon $\tau_0$,
	\item $1-\zeta_g(1-\gamma(\varep)+\rho\varep)\geq 0$
	\item $\lim_{t \to 0^+} t^{\zeta_g\varep}\|S(t)x_0\|_{X_{1+\varep}}=0$, by Proposition \ref{proposition S s.continuous and bounded growth in compacts}.
\end{itemize} 
Consider the set
$${\mathcal{K}(\tau_0,\mu)}=\left\lbrace \phi \in C\left((0,\tau_0];X_{1+\varep}\right);\, \|\phi\|_{\mathcal{K}(\tau_0,\mu)}:=\sup_{0<t\leq \tau_0}t^{\zeta_g\varep}\|\phi(t)\|_{X_{1+\varep}}\leq \mu\right\rbrace$$
and, fixed $u_0 \in B_{X_1}[x_0,r]=\{x \in X_1; \, \|x-x_0\|_{X_1}\leq r\}$, the map $T$ defined on ${\mathcal{K}(\tau_0,\mu)}$ by
$$(T \phi)(t)=S(t)u_0+\int_0^t S(t-s)f(s, \phi (s))ds, \quad t \in (0,\tau_0].$$
As one can check, $\mathcal{K}(\tau_0,\mu)$, equipped with the metric
$$d(\phi,\psi):=\|\phi-\psi\|_{\mathcal{K}(\tau_0,\mu)},$$
is a complete metric space. Our goal is to prove that $T$ is a contraction on ${\mathcal{K}(\tau_0,\mu)}$; so we can use the Banach fixed point theorem to obtain a function $u(\cdot\,;u_0) \in {\mathcal{K}(\tau_0,\mu)}$ satisfying
$$u(t;u_0)=S(t)u_0+\int_0^t S(t-s)f(s, u(s;u_0))ds, \quad t \in (0,\tau_0].$$
For this purpose, let us prove that if $\phi  \in {\mathcal{K}(\tau_0,\mu)}$, then $T\phi \in C((0,\tau_0];X_{1+\theta})$ for all $\theta \in [0,\gamma(\varep)-1+1/\zeta_g)$.  Indeed, Theorem \ref{Theorem delta S from gamma to 1+theta} guarantees us that $S(\cdot)u_0 \in C((0,\tau_0];X_{1+\theta})$. Then, let us consider the function $t \mapsto \int_0^t S(t-s)f(s,\phi(s))ds$. For $t_1,t_0 \in (0,\tau_0]$, say $t_1<t_0$, we have

\begin{align}
	&\left\|\int_{0}^{t_1}S(t_1-s)f(s,\phi(s))ds-\int_{0}^{t_0}S(t_0-s)f(s,\phi (s))ds\right\|_{X_{1+\theta}} \nonumber\\
	&\leq\left\|\int_{0}^{t_1}\left(S(t_1-s)-S(t_0-s)\right)f(s,\phi (s))ds\right\|_{X_{1+\theta}}\label{2s poi}
	\\
	&\quad+\left\|\int_{t_1}^{t_0}S(t_0-s)f(s,\phi (s))ds\right\|_{X_{1+\theta}}.\label{2t poi}
\end{align}
The term in \eqref{2s poi} satisfies

\begin{align}\label{qaqa}
	&\left\|\int_{0}^{t_1}\left(S(t_1-s)-S(t_0-s)\right)f(s,\phi (s))ds\right\|_{X_{1+\theta}}\nonumber\\ \leq&\, \int_{0}^{t_1}\|\left(S(t_1-s)-S(t_0-s)\right)\|_{\mathcal{B}(X_{\gamma(\varep)},X_{1+\theta})}\|f(s,\phi (s))\|_{X_{\gamma(\varep)}}ds\nonumber\\
	\leq&\,
	Mc\int_{0}^{t_1}\ln\left(\frac{t_0-s}{t_1-s}\right)^{\gamma(\varep)-\theta}(t_1-s)^{-\zeta_g(1+\theta-\gamma(\varep))} \left(\|\phi(s)\|_{X_{1+\varep}}^{\rho}+\nu(s)s^{-1+\zeta_g(1-\gamma(\varep))}\right)ds\nonumber\\
	\leq&\, 
	M c \int_{0}^{t_1}\ln\left(\frac{t_0-s}{t_1-s}\right)^{\gamma(\varep)-\theta}(t_1-s)^{-\zeta_g(1+\theta-\gamma(\varep))}\left(\left(\mu s^{-\zeta_g\varep}\right)^{\rho}+\nu(\tau_0)s^{-1+\zeta_g(1-\gamma(\varep))}\right)ds
	\nonumber\\
	=&\,
	M c \mu^\rho \int_{0}^{t_1}\ln\left(\frac{t_0-s}{t_1-s}\right)^{\gamma(\varep)-\theta}(t_1-s)^{-\zeta_g(1+\theta-\gamma(\varep))}s^{-\zeta_g\varep\rho}ds\nonumber\\
	&+
	M c \nu(\tau_0) \int_{0}^{t_1}\ln\left(\frac{t_0-s}{t_1-s}\right)^{\gamma(\varep)-\theta}(t_1-s)^{-\zeta_g(1+\theta-\gamma(\varep))}s^{-1+\zeta_g(1-\gamma(\varep))}ds\nonumber\\
	=&\,
	M c \mu^\rho t_1^{1-\zeta_g(1+\theta-\gamma(\varep)+\varep\rho)} \int_{0}^{1}\ln\left(\frac{t_0/t_1-s}{1-s}\right)^{\gamma(\varep)-\theta}(1-s)^{-\zeta_g(1+\theta-\gamma(\varep))}s^{-\zeta_g\varep\rho}ds\nonumber\\
	&+
	M c \nu(\tau_0) t_1^{\zeta_g\theta} \int_{0}^{1}\ln\left(\frac{t_0/t_1-s}{1-s}\right)^{\gamma(\varep)-\theta}(1-s)^{-\zeta_g(1+\theta-\gamma(\varep))}s^{-1+\zeta_g(1-\gamma(\varep))}ds,
\end{align}
 which, by Lemma \ref{Lemma I kappa goes to zero as kappa goes to 1+}, goes to $0$ as $|t_1-t_0|\to 0$. While the term in \eqref{2t poi} satisfies

\begin{align}\label{qsqs}
	&\left\|\int_{t_1}^{t_0}S(t_0-s)f(s,\phi (s))ds\right\|_{X_{1+\theta}}\nonumber\\
	\leq&\,
	\int_{t_1}^{t_0}M(t_0-s)^{-\zeta_g(1+\theta-{\gamma(\varep)})}c\left(\|\phi (s)\|_{X_{1+\varep}}^\rho+\nu(s)s^{-1+\zeta_g(1-\gamma(\varep))}\right)ds \nonumber\\
	\leq&\,
	Mc \int_{t_1}^{t_0}(t_0-s)^{-\zeta_g(1+\theta-{\gamma(\varep)})}\left(\left(s^{-\zeta_g\varep}\mu\right)^\rho+\nu(\tau_0) s^{-1+\zeta_g(1-\gamma(\varep))}\right)ds\nonumber\\
	=&\,
	Mc\mu^\rho\int_{t_1}^{t_0}(t_0-s)^{-\zeta_g(1+\theta-{\gamma(\varep)})}s^{-\zeta_g\varep\rho}ds+
	Mc\nu(\tau_0)\int_{t_1}^{t_0}(t_0-s)^{-\zeta_g(1+\theta-{\gamma(\varep)})}s^{-1+\zeta_g(1-\gamma(\varep))}ds\nonumber\\
	\leq&\,
	Mc\mu^\rho t_0^{1-\zeta_g(1+\theta-{\gamma(\varep)}+\varep\rho)}\int_{t_1/t_0}^{1}(1-s)^{-\zeta_g(1+\theta-{\gamma(\varep)})}s^{-\zeta_g\varep\rho}ds\nonumber\\
	&+
	Mc\nu(\tau_0) t_0^{-\zeta_g\theta}\int_{t_1/t_0}^1(1-s)^{-\zeta_g(1+\theta-{\gamma(\varep)})}s^{-1+\zeta_g(1-\gamma(\varep))}ds
\end{align}
which also goes to $0$ as $|t_1-t_0|\to0$. Therefore,

\begin{equation}\label{equation hjt}
	T\phi \in C((0,\tau_0];X_{1+\theta}), \quad \textrm{for all } \phi \in {\mathcal{K}(\tau_0,\mu)} \textrm{ and } \theta \in [0,\gamma(\varep)-1+1/\zeta_g),
\end{equation}
and, in particular,
\begin{equation}\label{plk}
	T\phi\in C((0,\tau_0];X_{1+\varep}), \quad \textrm{for all }\phi \in {\mathcal{K}(\tau_0,\mu)}.
\end{equation}
Moreover, for all $t \in (0,\tau_0]$, by Lemma \ref{lemma critical case}, \eqref{gyg} and \eqref{gvg}, we have
\begin{align*}
	t^{\zeta_g\varep}\left\|\int_{0}^{t}S(t-s)f(s,\phi (s))ds\right\|_{X_{1+\varep}}&
	\leq
	m(\varep,0,\tau_0)\left(\left(\sup_{0<s\leq t}s^{\zeta_g\varep}\|\phi(s)\|_{X_{1+\varep}}\right)^\rho t^{1-\zeta_g(1-\gamma(\varep)+\rho\varep)} +\nu(t)\right)\\
&	\leq
	m(\varep,0,\tau_0)\mu^{\rho} t^{1-\zeta_g(1-\gamma(\varep)+\rho\varep)} +m(\varep,0,\tau_0)\nu(\tau_0)\\
&	\leq \frac{\mu}{4} + \frac{\mu}{4}=\frac{\mu}{2},
\end{align*}
while by \eqref{ggg} and \eqref{g g},
\begin{align*}
	t^{\zeta_g\varep}\|S(t)u_0\|_{X_{1+\varep}}\leq&\, t^{\zeta_g\varep}\|S(t)(u_0-x_0)\|_{X_{1+\varep}}+t^{\zeta_g\varep}\|S(t)x_0\|_{X_{1+\varep}}\\
	\leq&\, M\|u_0-x_0\|_{X_1}+\mu/4\\
	\leq&\, Mr+\frac{\mu}{4} = \frac{\mu}{2}.
\end{align*}
Then
\begin{equation}\label{rty}
	\|T\phi\|_{\mathcal{K}(\tau_0,\mu)} =\sup_{0<t\leq \tau_0}t^{\zeta_g\varep}\|(T\phi)(t)\|_{X_{1+\varep}} \leq \frac{\mu}{2}+\frac{\mu}{2}=\mu.
\end{equation}
Therefore, \eqref{plk} and \eqref{rty} mean that $T\left({\mathcal{K}(\tau_0,\mu)}\right)\subset {\mathcal{K}(\tau_0,\mu)}$. To conclude the prove that $T$ is a contraction, we observe that, if $\phi,\psi \in {\mathcal{K}(\tau_0,\mu)}$, then Lemma \ref{lemma critical case} also provides
\begin{align*}
	&t^{\zeta_g\varep}\|(T\phi)(t)-(T\psi)(t)\|_{X_{1+\varep}}
	=\,t^{\zeta_g\varep}\left\|\int_{0}^t S(t-s)\left(f(s,\phi (s))-f(s,\psi(s))\right)ds\right\|_{X_{1+\varep}}\\
	\leq&\, 
	m(\varep,0,\tau_0)\left(2\mu^{\rho-1}t^{1-\zeta_g(1-{\gamma(\varep)}+\rho\varep)}+\nu(\tau_0)\right) \sup_{0<s\leq t}s^{\zeta_g\varep}\|\phi(s)-\psi(s)\|_{X_{1+\varep}}\\
	\leq&\, \frac{1}{2}\sup_{0<s\leq t}s^{\zeta_g\varep}\|\phi(s)-\psi(s)\|_{X_{1+\varep}}.
\end{align*}
Thus, $$\|T\phi-T\psi\|_{\mathcal{K}(\tau_0,\mu)}=\sup_{0<s\leq \tau_0}t^{\zeta_g\varep}\|(T\phi)(t)-(T\psi)(t)\|_{X_{1+\varep}}\leq \frac{1}{2}\|\phi-\psi\|_{\mathcal{K}(\tau_0,\mu)}.$$
We can, therefore, evoke the Banach fixed point theorem which guarantees us the existence of a function $\phi_0 \in {\mathcal{K}(\tau_0,\mu)}$ such that
$$\phi_0(t)=S(t)u_0+\int_0^t S(t-s)f\left(s,\phi_0(s)\right)ds, \quad t \in (0,\tau_0].$$
So, let us define $u(t;u_0)=\phi_0(t)$ for $t \in (0,\tau_0]$ and $u(0,u_0)=u_0$.

Now, suppose $\theta \in (0,\gamma(\varep)-1+1/\zeta_g)$. By Lemma \ref{lemma critical case},
\begin{align}
	&t^{\zeta_g\theta}\left(\|u(t;u_0)\|_{X_{1+\theta}}-\|S(t)u_0\|_{X_{1+\theta}}\right)\leq t^{\zeta_g\theta}\left\|\int_{0}^{t}S(t-s)f(s,u(s;u_0))ds\right\|_{X_{1+\theta}}\nonumber\\
	&\leq
	m(\theta,0,\tau_0)\left[\left(\sup_{0<s\leq t}s^{\zeta_g\varep}\|u(s;u_0)\|_{X_{1+\varep}}\right)^\rho t^{1-\zeta_g(1-\gamma(\varep)+\rho\varep)}+\nu(t)\right].\label{equation pol}
\end{align}
In particular, if $\theta = \varep$, we have
\begin{align*}
	&t^{\zeta_g\varep}\|u(t;u_0)\|_{X_{1+\varep}}-t^{\zeta_g\varep}\|S(t)u_0\|_{X_{1+\varep}}\leq 
	\\
	&
	m(\varep,0,\tau_0)\left(\left(\sup_{0<s\leq t}s^{\zeta_g\varep}\|u(s;u_0)\|_{X_{1+\varep}}\right)^\rho t^{1-\zeta_g(1-\gamma(\varep)+\rho\varep)} +\nu(t)\right)\\
	\leq &\,
	m(\varep,0,\tau_0)\mu^{\rho-1}\left(\sup_{0<s\leq t}s^{\zeta_g\varep}\|u(s;u_0)\|_{X_{1+\varep}}\right)t^{1-\zeta_g(1-{\gamma(\varep)}+\rho\varep)} +m(\varep,0,\tau_0)\nu(t) \\
	\leq &\,
	\frac{1}{4}\sup_{0<s\leq t}s^{\zeta_g\varep}\|u(s;u_0)\|_{X_{1+\varep}} +m(\varep,0,\tau_0)\nu(t).
\end{align*}
Then,
\begin{align*}
	\sup_{0<s\leq t}s^{\zeta_g\varep}\|u(s;u_0)\|_{X_{1+\varep}}\leq	&\sup_{0<s\leq t}s^{\zeta_g\varep}\|S(s)u_0\|_{X_{1+\varep}}\\
	&+\frac{1}{4}\sup_{0<s\leq t}s^{\zeta_g\varep}\|u(s;u_0)\|_{X_{1+\varep}} +m(\varep,0,\tau_0)\nu(t)
\end{align*}
and
\begin{align*}
	\frac{3}{4}\sup_{0<s\leq t}s^{\zeta_g\varep}\|u(s;u_0)\|_{X_{1+\varep}}\leq\sup_{0<s\leq t}s^{\zeta_g\varep}\|S(s)u_0\|_{X_{1+\varep}}+m(\varep,0,\tau_0)\nu(t).
\end{align*}
Therefore, if $J \subset B_{X_1}[x_0,r]$ is compact, we have
\begin{align*}
	&\frac{3}{4}\lim_{t\to0^+}\sup_{u_0 \in J} \left(\sup_{0<s\leq t}s^{\zeta_g\varep}\|u(s;u_0)\|_{X_{1+\varep}}\right)\\
	\leq&\,
	\lim_{t\to0^+}\sup_{u_0 \in J} \left(\sup_{0<s\leq t}s^{\zeta_g\varep}\|S(s)u_0\|_{X_{1+\varep}}\right)+\lim_{t\to0^+}m(\varep,0,\tau_0)\nu(t)\\
	=&\,
	\lim_{t\to0^+}\sup_{0<s\leq t} s^{\zeta_g\varep}\left(\sup_{u_0 \in J}\|S(s)u_0\|_{X_{1+\varep}}\right) \\
	=&\,
	\lim_{t\to0^+} s^{\zeta_g\varep}\left(\sup_{u_0 \in J}\|S(s)u_0\|_{X_{1+\varep}}\right).
\end{align*}
Proposition \ref{proposition S s.continuous and bounded growth in compacts} states that the last limit above is equal to $0$. Then,
$$\lim_{t\to0^+}\sup_{u_0 \in J} \left(\sup_{0<s\leq t}s^{\zeta_g\varep}\|u(s;u_0)\|_{X_{1+\varep}}\right)=0.$$
Using this information in \eqref{equation pol}, it follows that
\begin{equation}\label{equation hjy}
	\lim_{t\to0^+}t^{\zeta_g\theta}\left(\sup_{u_0 \in J}\|u(t;u_0)\|_{X_{1+\theta}}\right)=0.
\end{equation}

The only condition we still have to check in order to prove that $u(\cdot\,;u_0)$ is an $\varep$-regular mild solution to \eqref{abstract equation in introduction} is continuity at $t=0$. About that, we have
\begin{align*}
	&\|u(t;u_0)-u_0\|_{X_1}\leq \|S(t)u_0-u_0\|_{X_1}+\left\|\int_{0}^{t}S(t-s)f(s,u(s;u_0))ds\right\|_{X_{1}}\\
	\leq&\,
	\|S(t)u_0-u_0\|_{X_1} \\
	+& \,
	m(0,0,\tau_0)\left(\left(\sup_{0<s\leq t}s^{\zeta_g\varep}\|u(s;u_0)\|_{X_{1+\varep}}\right)^\rho t^{1-\zeta_g(1-\gamma(\varep)+\rho\varep)} +\nu(t)\right).
\end{align*}
which, by \eqref{equation hjy} and the strong continuity of $S$,  goes to $0$ as $t\to0^+$. We also observe that (A) is a consequence of \eqref{equation hjt} and \eqref{equation hjy}. 

Let us prove the other items.

\noindent\textbf{Proof of {(B)}:} If $u_0, v_0 \in B_{X_1}[x_0,r]$, it follows from Lemma \ref{lemma critical case} that
\begin{align}\label{fgt}
	&t^{\zeta_g\theta}\|u(t;u_0)-u(t;v_0)\|_{X_{1+\theta}}\leq  t^{\zeta_g\theta}\|S(t)(u_0-v_0)\|_{X_{1+\theta}}\nonumber\\
	&\qquad+t^{\zeta_g\theta}\left\|\int_0^t S(t-s)\left(f(s,u(s;u_0))-f(s,u(s,v_0))\right)ds\right\|_{X_{1+\theta}}\nonumber\\
	\leq&\, M\|u_0-v_0\|_{X_1}+m(\theta,0,\tau_0)\left(2\mu^{\rho-1}t^{1-\zeta_g(1-{\gamma(\varep)}+\rho\varep)}+\nu(\tau_0)\right)\nonumber\\
	&\qquad \qquad \qquad \times\sup_{0<s\leq t}s^{\zeta_g\varep}\|\phi(s)-\psi(s)\|_{X_{1+\varep}}.
\end{align}
In particular, for $\theta=\varep$,
\begin{align*}
	&t^{\zeta_g\varep}\|u(t;u_0)-u(t;v_0)\|_{X_{1+\varep}}\leq M\|u_0-v_0\|_{X_1}\nonumber\\
	&\quad+m(\varep,0,\tau_0)\left(2\mu^{\rho-1}t^{1-\zeta_g(1-{\gamma(\varep)}+\rho\varep)}+\nu(\tau_0)\right)\sup_{0<s\leq t}s^{\zeta_g\varep}\|\phi(s)-\psi(s)\|_{X_{1+\varep}}\\
	&\leq
	M\|u_0-v_0\|_{X_{1}} +\frac{3}{4}\sup_{0<s\leq t}s^{\zeta_g\varep}\|u(s;u_0) -u(s,v_0)\|_{X_{1+\varep}}.
\end{align*}
Then,
\begin{align*}
	&\sup_{0<s\leq t}s^{\zeta_g\varep}\|u(s;u_0) -u(s,v_0)\|_{X_{1+\varep}}\leq M\|u_0-v_0\|_{X_{1}} \\
	&+\frac{3}{4}\sup_{0<s\leq t}s^{\zeta_g\varep}\|u(s;u_0) -u(s,v_0)\|_{X_{1+\varep}}
\end{align*}
and so
$$\sup_{0<s\leq t}s^{\zeta_g\varep}\|u(s;u_0) -u(s,v_0)\|_{X_{1+\varep}}\leq4 M\|u_0-v_0\|_{X_{1}}.$$
Then, by \eqref{fgt}, we have
\begin{align*}
	t^{\zeta_g\theta}\|u(t;u_0)-u(t;v_0)\|_{X_{1+\theta}}\leq&\,
	M\|u_0-v_0\|_{X_1}\\
	&+m(\theta,0,\tau_0)\left(2\mu^{\rho-1}+\nu(\tau_0)\right)\left(4M\|u_0-v_0\|_{X_{1}} \right)\\
	\leq&\, L(\theta,x_0)\|u_0-v_0\|_{X_1},
\end{align*}
where
$$L(\theta,x_0)= 	M+m(\theta,0,\tau_0)\left(2\mu^{\rho-1}+\nu(\tau_0)\right)4M.$$

\noindent\textbf{Proof of {(C)}:} Let $\tau_1>0$ and $v:[0,\tau_1]\to X_1$ be an $\varep$-regular mild solution on $[0,\tau_1]$ for the Problem \eqref{abstract equation in introduction} starting in some $u_0 \in B_{X_1}[x_0,r]$ and satisfying
$$\lim_{t\to0^+} t^{\zeta_g\varep}\|v(t)\|_{X_{1+\varep}}=0.$$
Then, we can choose $\tau \in (0,\min\{\tau_1,\tau_0\}]$ small enough so that
$$\sup_{0<t\leq  \tau}t^{\zeta_g\varep}\|v(t)\|_{X_{1+\varep}}\leq \mu.$$
Therefore, the restrictions of the functions $v$ and $u(\cdot\,;u_0)$ to the interval $(0, \tau]$ belong to the set
$$\mathcal{K}(\tau,\mu):=\left\lbrace \phi \in C\left((0, \tau];X_{1+\varep
}\right);\, \sup_{0<t\leq  \tau}t^{\zeta_g\varep}\|\phi(t)\|_{X_{1+\varep}}\leq \mu\right\rbrace$$
and they are fixed points of the map $\widetilde{T}$ defined on $\mathcal{K}(\tau,\mu)$ by
$$(\widetilde{T}\phi)(t)=S(t)u_0+\int_0^t S(t-s)f(s, \phi (s))ds, \quad t \in (0, \tau].$$
We note that $\tau_0>0$ was taken arbitrarily small, so the proof that $\widetilde{T}$ is a contraction on $\mathcal{K}(\tau,\mu)$ is entirely analogous to the proof for $T$ on ${\mathcal{K}(\tau_0,\mu)}$. Then, by the uniqueness of fixed points for $\widetilde{T}$ on $\mathcal{K}(\tau,\mu)$, we have that $v(t)=u(t;u_0)$ for all $t \in (0, \tau]$. Actually, since $v(0)=u_0=u(0,u_0)$, the equality holds on $[0, \tau]$. On the other hand, denoting
$$e:=\sup_{s \in \left[ \tau, \min\{ \tau_1,\tau_0\}\right]}\left(\|v(s)\|_{X_{1+\varep}}^{\rho-1}+\|u(s;u_0)\|_{X_{1+\varep}}^{\rho-1}+\nu(s)s^{-1+\zeta_g(1-\gamma(\varep))}\right)<\infty,$$
for each $t \in [ \tau,\min\{ \tau_1,\tau_0\}]$, we have
\begin{align*}
	&\|v(t)-u(t;u_0)\|_{X_{1+\varep}}\leq \int_{0}^{t}\left\|S(t-s)f\left(s,v(s)\right)-S(t-s)f\left(s,u(s;u_0)\right)\right\|_{X_{1+\varep}}ds\\
	\leq&\, \int_{0}^{t}M(t-s)^{-\zeta_g(1+\varep-{\gamma(\varep)})}\|f(s,v(s))-f(s,u(s;u_0))\|_{X_{\gamma(\varep)}}ds\\
	\leq&\,
	Mc\int_{ \tau}^{t}(t-s)^{-\zeta_g(1+\varep-{\gamma(\varep)})}\|v(s)-u(s;u_0)\|_{X_{1+\varep}}\\
	&\times\left(\|v(s)\|_{X_{1+\varep}}^{\rho-1}+\|u(s;u_0)\|_{X_{1+\varep}}^{\rho-1}+\nu(s)s^{-1+\zeta_g(1-\gamma(\varep))}\right)ds\\
	\leq&\,
	Mce\int_{ \tau}^{t}(t-s)^{-\zeta_g(1+\varep-{\gamma(\varep)})}\|v(s)-u(s;u_0)\|_{X_{1+\varep}}ds.
\end{align*}
It follows from Singular Grönwall's Inequality \ref{Grönwall's inequality} that
$$\|v(t)-u(t;u_0)\|_{X_{1+\varep}}=0, \quad \forall \,t \in[\tau,\min\{\tau_1,\tau_0\}].$$
Therefore, $v(t)=u(t;u_0)$ for all $t \in [0,\min\{ \tau_1,\tau_0\}]$, as (C) states.
\\\\\noindent\textbf{Proof of {(D)}:} 
If $\zeta_g< (1-\gamma(\varep)+\rho\varep)^{-1}$, then
$$\lim_{t\to0^+}t^{1-\zeta_g(1-\gamma(\varep)+\rho\varep)}=0.$$
Hence, we can choose $\mu>0$ arbitrarily large (and $r=\mu/(8M)$ will be large as well) and $\tau_0 \in (0,1)$ small enough such that \eqref{gyg}, \eqref{ggg}, and \eqref{gvg} remain true.

\end{proof}
\subsection{Continuation}
If $u$ and $v$ are $\varep$-regular mild solutions for the problem \eqref{abstract equation in introduction}  satisfying
$$\lim_{t \to 0^+}t^{\zeta_g\varep}\|u(t)\|_{X_{1+\varep}}=\lim_{t \to 0^+}t^{\zeta_g\varep}\|v(t)\|_{X_{1+\varep}}=0,$$
then it follows from Theorem \ref{theorem critic case well-posedness} that $u=v$ on some sufficiently small interval $[0,\tau]$. However, we still do not know if a type of "bifurcation" is possible along the time of existence of such solutions, in the sense that, $u$ and $v$ exist and are distinct beyond $[0,\tau]$. We prove next that the answer to this question is no. Before, we give the following definition.

\begin{de}[Continuation]
	Let $u \in C([0,\tau];X_1)\cap C((0,\tau];X_{1+\varep})$ be a $\varep$-regular mild solution for the problem \eqref{abstract equation in introduction}. Given $\tau_1>\tau$, a function $v \in C([0,\tau_1];X_1)\cap C((0,\tau_1];X_{1+\varep})$ is said to be a \emph{continuation on $[0,\tau_1]$} of $u$ if $\widetilde{u}$ is also an $\varep$-regular mild solution to \eqref{abstract equation in introduction} and $\widetilde{u}(t)=u(t)$, for all $t \in [0,\tau]$.
\end{de}
\begin{theorem}[Continuation result]\label{theorem continuation critic case}
	Let $u$ be an $\varep$-regular mild solution for the problem \eqref{abstract equation in introduction} on an interval $[0,\tau]$, for some $\tau>0$. If $u$ satisfies
	$$\lim_{t \to 0^+}t^{\zeta_g\varep}\|u(t)\|_{X_{1+\varep}}=0,$$
	then there exist $\tau_1>\tau$ and a continuation $\bar{u} \in C([0,\tau_1];X_1)\cap C((0,\tau_1];X_{1+\varep})$ of $u$ that is unique on the interval $[0,\tau']$, for any $\tau' \in [\tau,\tau_1]$.	
\end{theorem}
\begin{proof}
Let us denote
$$\eta=\max\left\lbrace\sup_{0<s\leq \tau}s^{\zeta_g\varep}\|u(s)\|_{X_{1+\varep}}\,,\,\,\,1+(\tau+1)^{\zeta_g\varep}\|u(\tau)\|_{X_{1+\varep}}\right\rbrace.$$
Furthermore, as we have been doing, for each $\theta \in [0,\gamma(\varep)-1+1/\zeta_g)$ and $0\leq a <b<\infty$, we denote by $m(\theta,a,b)$ the biggest of the constants
$$Mc \textbf{B}_{1-a/b}(1-\zeta_g\varep\rho,1-\zeta_g(1+\theta-\gamma(\varep)))$$
and	$$Mc\textbf{B}_{1-a/b}\left(\zeta_g(1-\gamma(\varep)),1-\zeta_g(1+\theta-\gamma(\varep))\right).$$
We note that $m(\theta,a,b)\to 0$ as $a/b\to 1^{-}$.	Therefore, we can choose $\tau_1 \in (\tau,\tau+1]$ arbitrarily close of $\tau$ such that
\begin{equation*}\label{qqr}
	m(\varep,\tau,\tau_1)\left(\eta^\rho t^{1-\zeta_g(1-\gamma(\varep)+\rho\varep)} +\nu(t)\right)\leq \frac{1}{4}
\end{equation*}
and
\begin{equation*}\label{qqt}
	m(\varep,\tau,\tau_1)\left(2\eta^{\rho-1}t^{1-\zeta_g(1-{\gamma(\varep)}+\rho\varep)}+\nu(\tau_1)\right)\leq \frac{1}{2},
\end{equation*}
for $t \in [\tau,\tau_1]$. Further, we can show that
$$t^{\zeta_g\varep}\left\|\int_0^\tau\left(S(t-s)-S(\tau-s)\right)f(s,u(s))ds\right\|_{X_{1+\varep}}\to 0,$$
as $t\to \tau^+$, just performing the same computations we did in Theorem \ref{theorem critic case well-posedness} to estimate \eqref{2s poi}.
Hence we can also suppose $\tau_1 \in (\tau,\tau+1]$ is taken such that
\begin{equation*}\label{qqw}
	t^{\zeta_g\varep}\left\|\int_0^\tau\left(S(t-s)-S(\tau-s)\right)f(s,u(s))ds\right\|_{X_{1+\varep}}\leq \frac{1}{4}
\end{equation*}
and
\begin{equation*}\label{qqq}
	t^{\zeta_g\varep}\|S(t)u_0 - S(\tau)u_0\|_{X_{1+\varep}}\leq \frac{1}{2},
\end{equation*}
for $t \in [\tau,\tau_1]$.

\noindent{\bf Existence:}	Let $\mathcal{S}$ be the set of all functions $\phi \in C((0,\tau_1];X_{1+\varep})$ such that $\phi(t)=u(t)$, for $t \in (0,\tau]$, and
$$\sup_{t \in [\tau, \tau_1]} t^{\zeta_g\varep}\|\phi(t)-u(\tau)\|_{X_{1+\varep}}\leq 1.$$
For each $\phi \in \mathcal{S}$, we have
\begin{align*}
	&\sup_{t \in (0,\tau_1]}t^{\zeta_g\varep}\|\phi(t)\|_{X_{1+\varep}}\leq \max\left\lbrace \sup_{t \in (0,\tau]} t^{\zeta_g\varep}\|\phi(t)\|_{X_{1+\varep}}\,,\, \sup_{t \in [\tau,\tau_1]} t^{\zeta_g\varep}\|\phi(t)\|_{X_{1+\varep}}\right\rbrace\\
	\leq & \max\left\lbrace\sup_{t \in (0,\tau]}t^{\zeta_g\varep}\|u(t)\|_{X_{1+\varep}}\,,\,\,\,\sup_{t \in [\tau, \tau_1]} t^{\zeta_g\varep}\|\phi(t)-u(\tau)\|_{X_{1+\varep}}+ \tau_1^{\zeta_g\varep}\|u(\tau)\|_{X_{1+\varep}}\right\rbrace\\
	\leq& \max\left\lbrace\sup_{t \in (0,\tau]}t^{\zeta_g\varep}\|u(t)\|_{X_{1+\varep}}\,,\,\,\,1+(\tau+1)^{\zeta_g\varep}\|u(\tau)\|_{X_{1+\varep}}\right\rbrace=\eta.
\end{align*}
Thereby, $\mathcal{S}\subset \mathcal{K}(\tau_1,\eta)$. It is not difficult to verify that $\mathcal{S}$ is closed in $\mathcal{K}(\tau_1,\eta)$,  and hence $\mathcal{S}$ is complete under the metric
$$d(\phi,\psi):=\|\phi-\psi\|_S := \|\phi-\psi\|_{\mathcal{K}(\tau_1,\eta)}=\sup_{0<t\leq  \tau_1}t^{\zeta_g\varep}\|\phi(t)-\psi(t)\|_{X_{1+\varep}}.$$
On $\mathcal{S}$, consider the map $T$ given by
$$(T\phi)(t)=S(t)u_0+\int_{0}^t S(t-s)f(s,\phi(s))ds, \quad t \in (0,\tau_1].$$
Analogously to what we did in Theorem \ref{theorem critic case well-posedness}, one can prove that $T\phi \in C((0,\tau_1];X_{1+\theta}))$, for all $\phi \in \mathcal{K}(\tau_1,\eta)$ and any $\theta \in [0,\gamma(\varep)-1+1/\zeta_g)$. In particular, if $\phi \in \mathcal{S}$, then $T\phi \in C((0,\tau_1];X_{1+\varep}))$. Moreover, for $t \in (0,\tau]$, we have $(T\phi)(t)=(Tu)(t)=u(t)$. Thus, to prove that $T\phi \in \mathcal{S}$, we just need to check that
\begin{equation}\label{dre}
	\sup_{t \in [\tau, \tau_1]} t^{\zeta_g\varep}\|(T\phi)(t)-u(\tau)\|_{X_{1+\varep}}\leq 1.
\end{equation}
Given $t \in [\tau,\tau_1]$, by Lemma \ref{lemma critical case}, we have
\begin{align*}
	&t^{\zeta_g\varep}\|(T\phi)(t)-u(\tau)\|_{X_{1+\varep}}\leq t^{\zeta_g\varep}\|S(t)u_0 - S(\tau)u_0\|_{X_{1+\varep}}\\
	&+t^{\zeta_g\varep}\left\| \int_0^t S(t-s)f(s,\phi(s))ds+\int_0^\tau S(\tau-s)f(s,u(s))ds\right\|_{X_{1+\varep}} \\
	& \leq \, \frac{1}{2}+t^{\zeta_g\varep}\left\|\int_0^\tau\left(S(t-s)-S(\tau-s)\right)f(s,u(s))ds\right\|_{X_{1+\varep}}+t^{\zeta_g\varep}\left\| \int_\tau^t S(t-s)f(s,\phi(s))ds\right\|_{X_{1+\varep}}\\
	&\leq\, \frac{1}{2}+\frac{1}{4}
	+m(\varep,\tau,\tau_1)\left(\left(\sup_{0<s\leq t}s^{\zeta_g\varep}\|\phi(s)\|_{X_{1+\varep}}\right)^\rho t^{1-\zeta_g(1-\gamma(\varep)+\rho\varep)} +\nu(t)\right)\\
	&\leq\, \frac{1}{2}+\frac{1}{4}
	+m(\varep,\tau,\tau_1)\left(\eta^\rho t^{1-\zeta_g(1-\gamma(\varep)+\rho\varep)} +\nu(t)\right)	\leq 1,
\end{align*}
from which \eqref{dre} follows. Therefore, $T(\mathcal{S})\subset \mathcal{S}$.

Now, if $\phi , \psi \in \mathcal{S}$ and $t \in [\tau,\tau_1]$, then
\begin{align*}
	&t^{\zeta_g\varep}\|(T\phi)(t)-(T\psi)(t)\|_{X_{1+\varep}}=t^{\zeta_g\varep}\left\|\int_{\tau}^t S(t-s)\left(f(s,\phi (s))-f(s,\psi(s))\right)ds\right\|_{X_{1+\varep}} \\
    &\leq\,
	m(\varep,\tau,\tau_1)\left(2\eta^{\rho-1}t^{1-\zeta_g(1-{\gamma(\varep)}+\rho\varep)}+\nu(\tau_1)\right) \sup_{0<s\leq t}s^{\zeta_g\varep}\|\phi(s)-\psi(s)\|_{X_{1+\varep}}\\
	&\leq\,
	\frac{1}{2}\sup_{0<s\leq t}s^{\zeta_g\varep}\|\phi(s)-\psi(s)\|_{X_{1+\varep}}.
\end{align*}
Consequently
$$\|T\phi-T\psi\|_S \leq \frac{1}{2}\|\phi - \psi\|_\mathcal{S}.$$
Then $T$ is a strict contraction on $\mathcal{S}$ and, by the Banach fixed point theorem, it has a unique fixed point $\phi_0 \in \mathcal{S}$. The function $\widetilde{u}$ defined by $\widetilde{u}(t)=\phi_0(t)$, for $t \in (0,\tau_1]$, and $\widetilde{u}(0)=u_0$, is a continuation of $u$ on $[0,\tau_1]$.

\noindent{\bf Uniqueness:} 
Let us suppose $u$ admits another continuation $v$ which is defined on some sub-interval $[0,\tau']\subset[0,\tau_1]$ and let
$$e:=\sup_{s \in [\tau,\tau']} \left(\|v(s)\|_{X_{1+\varep}}^{\rho-1}+\|\widetilde{u}(s)\|_{X_{1+\varep}}^{\rho-1}+\nu(s)s^{-1+\zeta_g(1-\gamma(\varep))}\right).$$
For each $t \in [\tau
,\tau']$, we have
\begin{align*}
	&\|v(t)-\widetilde{u}(t)\|_{X_{1+\varep}}\leq \int_{0}^{t}\left\|S(t-s)f\left(s,v(s)\right)-S(t-s)f\left(s,\widetilde{u}(s)\right)\right\|_{X_{1+\varep}}ds\\
	\leq&\, \int_{0}^{t}M(t-s)^{-\zeta_g(1+\varep-{\gamma(\varep)})}\|f(s,v(s))-f(s,\widetilde{u}(s))\|_{X_{\gamma(\varep)}}ds\\
	\leq&\,
	Mc\int_{ \tau}^{t}(t-s)^{-\zeta_g(1+\varep-{\gamma(\varep)})}\|v(s)-\widetilde{u}(s)\|_{X_{1+\varep}}\left(\|v(s)\|_{X_{1+\varep}}^{\rho-1}+\|\widetilde{u}(s)\|_{X_{1+\varep}}^{\rho-1}+\nu(s)s^{-1+\zeta_g(1-\gamma(\varep))}\right)ds\\
	\leq&\,
	Mce\int_{ \tau}^{t}(t-s)^{-\zeta_g(1+\varep-{\gamma(\varep)})}\|v(s)-\widetilde{u}(s)\|_{X_{1+\varep}}ds.
\end{align*}
Hence, the uniqueness follows from Singular Grönwall's Inequality, Lemma \ref{Grönwall's inequality}.
\end{proof}

\begin{remark}
	We note that the result above solves the issue of possible bifurcations. Indeed, let $u$ and $v$ be $\varep$-regular mild solutions defined on $[0,\tau(u)]$ and $[0,\tau(v)]$, respectively, satisfying
	$$\lim_{t \to 0^+}t^{\zeta_g\varep}\|u(t)\|_{X_{1+\varep}}=\lim_{t \to 0^+}t^{\zeta_g\varep}\|v(t)\|_{X_{1+\varep}}=0.$$
	We claim that $u=v$ on $[0,\min\{\tau(u),\tau(v)\}]$. Indeed, suppose not. Then, denoting
	$$I:=\{t \in [0,\min\{\tau(u),\tau(v)\}];\, u=v \textrm{ on }[0,t]\}, \quad \tau:=\max(I),$$
	and $w(t):=u(t)=v(t),$ for $t \in [0,\tau]$, we have that
	$$0<\tau<\min\{\tau(u),\tau(v)\}$$
	and $w$ is an $\varep$-regular mild solution on $[0,\tau]$ for the problem \eqref{abstract equation in introduction}. From Theorem \ref{theorem continuation critic case}, there is $\tau_1>\tau$ and a continuation $\bar{w}$ on $[0,\tau_1]$ of $w$. Let
	$\tau'=\min\{\tau_1,\tau(u),\tau(v)\}.$
	Since $\tau'>\tau$, we have that $\tau' \notin I$. On the other hand, by Theorem \ref{theorem continuation critic case}, $\bar{w}$ is the unique continuation on $[0,\tau']$ of $w$. Then, $u=\bar{w}=v \textrm{ on }[0,\tau']$, whence $\tau' \in I$, which is a contradiction.
\end{remark}

\begin{remark}
In this subsection, we addressed only the continuation of solutions in the critical case. Nevertheless, with some adjustments, a similar result can be developed in the {subcritical} framework. 
\end{remark}

\subsection{Blow-up alternative}

If $u$ is an $\varep$-regular mild solution on some interval $[0,\tau_0]$ to  \eqref{abstract equation in introduction}, then, by Theorem \ref{theorem continuation critic case}, the interval
$$\{ \tau> \tau_0\,; u \textrm{ admits a continuation on }[0,\tau] \}$$
is not empty. Hence, the following definition is consistent.

\begin{de}[Maximal time of existence]
	Let $u$ be an $\varep$-regular mild solution on $[0,\tau_0]$ to the Problem \eqref{abstract equation in introduction}. The number $\tau_{max} \in (\tau_0,\infty]$ defined by
	$$\tau_{max} := \sup \{ \tau> \tau_0\,; u \textrm{ admits a continuation on }[0,\tau] \} $$
	is said to be its {maximal time of existence}.
\end{de}

The next theorem is our {blow-up alternative} to \eqref{abstract equation in introduction}.
\begin{theorem}[Blow-up alternative]
	Let $u$ be the $\varep$-regular mild solution for the problem \eqref{abstract equation in introduction} that satisfies
	$$\lim_{t \to 0^+}t^{\zeta_g\varep}\|u(t)\|_{X_{1+\varep}}=0.$$ If its maximal time of existence $\tau_{\textrm{max}}>0$ is finite, then
	$$\limsup_{t \to \tau_{\textrm{max}}\,^-}\|u(t)\|_{X_{1+\varep}}= \infty.$$
\begin{proof}
Let us suppose by contradiction that $\tau_\textrm{max}<\infty$ and  $\limsup_{t \to \tau_{\textrm{max}}\,^-}\|u(t)\|_{X_{1+\varep}}< \infty.$
Then,
$$Q:=\sup_{t \in (0,\tau_{\textrm{max}})} t^{\zeta_g\varep}\|u(t)\|_{X_{1+\varep}}<\infty.$$ 
Let $\{t_n\}_{n \in \mathbb{N}} \subset (0,\tau_{\textrm{max}})$ be a sequence with $\lim_{n \to \infty}t_n = \tau_{\textrm{max}}.$ We claim that $\{u(t_n)\}_{n \in \mathbb{N}} \subset X_{1+\varep}$ is a Cauchy sequence. In fact, let $m,n$ in $\mathbb{N}$; without loss of generality, let us suppose $t_m<t_n$. Then
\begin{align*}
			&\|u(t_m)-u(t_n)\|_{X_{1+\varep}}\leq \|S(t_m)u_0-S(t_n)u_0\|_{X_{1+\varep}}\nonumber\\
			&+\left\|\int_0^{t_m}S(t_m-s)f(s,u(s))ds-\int_0^{t_n}S(t_n-s)f(s,u(s))ds\right\|_{X_{1+\varep}}\nonumber\\
			&\leq \,
			\|S(t_m)u_0-S(t_n)u_0\|_{X_{1+\varep}}
			+ \left\|\int_0^{t_m}\left(S(t_m-s)-S(t_n-s)\right)f(s,u(s))ds\right\|_{X_{1+\varep}}\nonumber\\
			&+
			\left\|\int_{t_m}^{t_n}S(t_n-s)f(s,u(s))ds\right\|_{X_{1+\varep}}.
\end{align*}
By Theorem \ref{Theorem delta S from gamma to 1+theta}, the first summand above goes to $0$ as $m,n \to \infty$. Furthermore, doing the same computations as we did in \eqref{qaqa} and \eqref{qsqs}, in the proof of Theorem \ref{theorem critic case well-posedness}, one can show that the second and third summands also go to $0$ as $m,n\to \infty$. Therefore, $\{u(t_n)\}_{n \in \mathbb{N}} \subset X_{1+\varep}$ is indeed a Cauchy sequence.

The arbitrariness of $\{t_n\}_{n \in \mathbb{N}}$ ensures that $\lim_{t \to \tau_{\textrm{max}}\,\!\!^-} u(t)$ exists in $X_{1+\varep}$, so we can extend $u$ on $[0,\tau_{\textrm{max}}]$ defining $$u(\tau_{\textrm{max}})= \lim_{t \to \tau_{\textrm{max}}\,^-}u(t),$$
which contradicts the Theorem \ref{theorem continuation critic case}.

\end{proof}
\end{theorem}

\begin{remark}
As in the previous subsection, we have focused exclusively on blow-up of solutions in the critical case. However, a similar result can be developed in the {subcritical} setting, as well.
\end{remark}

\section{Applications}\sectionmark{Applications}\label{chapter aplications}
This section focuses on applying our abstract results to specific problems. Although we will consider these applications in the critical setting, we want to emphasize that our study on the subcritical framework can be applied to the same examples. The prototype of material function we consider has the form
\begin{equation}\label{kernelappendix}
	g(t)=\sum_{i=1}^n k_it^{\alpha_i -1}e^{c_i t}, \quad t>0,
\end{equation}
where $k_i>0, \alpha_i>0, c_i \in \mathbb{R}.$ We mention that many types of materials, as the {Hookean solid}, the {Maxwell fluid}, the {Poynting-Thompson solid}, and the {power type materials}, are described by material functions of the form \eqref{kernelappendix}, see \cite[Chapter 5]{pruss2013evolutionary} with $a(t)=(1\ast g)(t)$, $t>0$. In the Appendix we prove that this type of material function verifies the conditions of Theorem \ref{generation + regularity}.

We recall the definition of fractional order Sobolev spaces, also called spaces of Bessel potentials, that will be used in the following subsections. Let \(\Omega\) be a bounded, smooth domain in \(\mathbb{R}^N\). For $l>0$ and $1 \leq p \leq \infty$, we denote by $W^{l,p}(\Omega)$ the Sobolev space of order $l$. It can be defined via complex interpolation as
$$
W^{l,p}(\Omega)
:=
[L^p(\Omega), W^{m,p}(\Omega)]_{l/m},
$$
 where $m$ is the smallest integer strictly greater than $l$, and $W^{m,p}(\Omega)$ denotes the standard Sobolev space of integer order $m$.
We also recall the well known Sobolev embeddings 
\begin{equation}\label{sobolev embeddings}
	W^{l_1,p_1}(\Omega)\hookrightarrow W^{l_2,p_2}(\Omega), \qquad \quad \frac{l_1}{N}-\frac{1}{p_1}\geq \frac{l_2}{N}-\frac{1}{p_2}, \quad 1<p_1\leq p_2<\infty.
\end{equation}
We denote by $W^{l,p}(\Omega;\R^N)$ the Sobolev spaces of vector-valued functions and, for the sake of simplicity, write $W^{l,p}(\Omega)$ in the scalar case.

\subsection[Navier-Stokes equations with hereditary viscosity]{Navier-Stokes equations with hereditary viscosity%
	\sectionmark{Navier-Stokes}}
\sectionmark{Navier-Stokes}

Let $\Omega \subset \mathbb{R}^N$, $N\geq 3$, be a bounded smooth domain and consider the  Navier-Stokes system with {hereditary viscosity}
\begin{equation}\label{problem ns}
	\left\{
	\begin{array}{lll}
		u_t + (u\cdot \nabla) u - \dint_0^t g(t-s)\Delta u(x,s)\,ds = -\nabla p + h, \quad \textrm{in } \Omega\times(0,\infty), \\
		\mathrm{div}(u)=0 \quad \textrm{in } \Omega\times (0,\infty), \\
		u(x,t)=0, \quad (x,t) \in \partial \Omega\times (0,\infty), \\
		u(x,0)=u_0(x), \quad x \in \Omega.
	\end{array}\right.
\end{equation}
Here, $u:\Omega \times (0,\infty) \to \mathbb{R}^N$ describes the velocity field of a fluid contained in $\Omega$, $p$ is the fluid pressure, $h$ is an external force, $u_0:\Omega \to \mathbb{R}^N$ is the velocity field at the initial time $t=0$, and $g:[0,\infty)\to \mathbb{R}$ represents a {material function}. Such models arise in the dynamics of non-Newtonian fluids or as viscoelastic models for the dynamics of turbulence statistics in Newtonian fluids and also is studied in \cite{barbu2003navier, deACD, andrade2021viana2021silva, mohan2019} by different approaches.

It is well known that, for $1<q<\infty$, the following algebraic and topological direct sum decomposition holds:
$$
L^q(\Omega;\mathbb{R}^N)
=
L^q_\sigma(\Omega;\mathbb{R}^N)
\oplus
\{\nabla \pi \,;\,\pi \in W^{1,q}_{\mathrm{loc}}(\Omega;\mathbb{R}), \, \nabla \pi \in L^{q}(\Omega;\mathbb{R}^N)\},
$$
where $L^q_\sigma(\Omega;\mathbb{R}^N)$ denotes the closure of
$$
\{u \in C^\infty_0(\Omega;\mathbb{R}^N) \,;\, \mathrm{div}(u)=0\}
$$
in $L^q(\Omega;\mathbb{R}^N)$, see \cite{fujiwara1977helmholtz}.
This is the so-called {Helmholtz decomposition} of $L^q(\Omega;\mathbb{R}^N)$.
Let $P$ denote the projection operator from $L^q(\Omega;\mathbb{R}^N)$ onto $L^q_\sigma(\Omega;\mathbb{R}^N)$. Applying $P$ to \eqref{problem ns}, we obtain
$$
 u_t + P(u\cdot \nabla) u
-
\int_0^t g(t-s) P\Delta u(x,s)\,ds
=
Ph(x,t),
\quad (x,t) \in \Omega\times(0,\infty).
$$
Setting $L=-P\Delta$, $u(t)=u(\cdot,t)$, and
$f(t,u(t))=-P(u(t)\cdot\nabla u(t))+Ph(\cdot,t)$,
we rewrite \eqref{problem ns} as
$$
u'(t)
=
\int_0^t g(t-s)Lu(s)\,ds
+
f(t,u(t)),
\quad t>0.
$$
The linear operator $L$ is the {Stokes operator}, and its domain is given by
$$
\mathcal{D}(L)
=
W^{2,q}(\Omega;\mathbb{R}^N)
\cap
W^{1,q}_0(\Omega;\mathbb{R}^N)
\cap
L^q_\sigma(\Omega;\mathbb{R}^N).
$$
On $L^q_\sigma(\Omega;\mathbb{R}^N)$, the operator $L$ is sectorial of angle $0$ and satisfies $0 \in \rho(L)$; see
\cite{giga1981analyticity,ladyzhenskaya1969viscous}. 
Its associated scale of {fractional power spaces} $\{E^\alpha_q\}_{\alpha \in \mathbb{R}}$ is preserved by the complex family of interpolation functors, that is,
\begin{equation}\label{BIP property}
	E_q^{(1-\theta)\alpha+\theta\beta}
	=
	[E^\alpha_q,E_q^\beta]_{\theta},
	\quad
	-\infty<\alpha\leq \beta<\infty,
	\ \theta \in (0,1),
\end{equation}
see \cite{giga1985domains,seeley1971norms} and \cite[Theorem~1.5.4]{amann1995linear}, and it satisfies the following duality property:
$$
\left(E^\alpha_q\right)^\ast
=
E^{-\alpha}_{q'},
$$
see \cite[Theorem~1.5.12]{amann1995linear}, where $q'$ denotes the conjugate exponent of $q$.
Furthermore, for $\alpha \geq 0$, one has the continuous embedding
$$
E_q^{\alpha}
\hookrightarrow
W^{2\alpha,q}(\Omega;\mathbb{R}^N),
$$
which, by \eqref{sobolev embeddings}, implies
$$
E^\alpha_q
\hookrightarrow
L^r(\Omega;\mathbb{R}^N),
\qquad
\text{if }\ 
\frac{2\alpha}{N}-\frac{1}{q}
\geq
-\frac{1}{r},
\quad
q\leq r<\infty.
$$
In particular, choosing
\begin{equation}\label{EquationExponentR}
	r=\frac{Nq}{N-2\alpha q},
\end{equation}
we obtain
\begin{equation}\label{E^a_q Embedding_posivite}
	E_q^{\alpha}
	\hookrightarrow
	L^r(\Omega;\mathbb{R}^N),
	\quad
	\text{for }
	0\leq \alpha <\frac{N}{2q}.
\end{equation}
By duality, one also derives that
$L^{r'}(\Omega;\mathbb{R}^N)\hookrightarrow E_{q'}^{-\alpha}$.
Since
$$
r'=\frac{Nq'}{N+2\alpha q'},
$$
by replacing $r', q'$, and $\alpha$ with $r, q$, and $-\alpha$, respectively, we obtain
\begin{equation}\label{E^a_q Embedding_negative}
	E_{q}^{\alpha}
	\hookleftarrow
	L^{r}(\Omega;\mathbb{R}^N),
	\quad
	\text{for }
	-\frac{N}{2q'}< \alpha\leq 0,
\end{equation}
where $r$ is still given by \eqref{EquationExponentR}.

We assume that the kernel $g$ satisfies the hypotheses of Theorem~\ref{generation + regularity} for some $\zeta_g>1$ and set
$$
X_0=E_q^{\frac{N-q}{2q}-\frac{1}{\zeta_g}}.
$$
Moreover, we denote by $A$ the {realization} of $-L$ in $X_0$. Then, $\sigma(A)=\sigma(-L)$ and
$-A\colon X_1=E_q^{\frac{N+q}{2q}-\frac{1}{\zeta_g}}\subset X_0 \to X_0$
is a sectorial operator of angle $0$; see \cite[Theorem~2.1.3]{amann1995linear}.
As before, for $k\in \mathbb{N}$, we denote by $X_k$ the space $\mathcal{D}(A^k)$ equipped with the norm
$\|x\|_{X_k}:=\|(I-A)^k x\|_{X_0}$.
Since $0 \in \rho(A)$, one can show that
$\|(I-A)^k \cdot\|_{X_0}\sim \|A^k \cdot\|_{X_0}$,
from which it follows that
$X_k=E_q^{k+\frac{N-q}{2q}-\frac{1}{\zeta_g}}$
with equivalent norms.
Then, by \eqref{BIP property}, taking $\{\mathcal{F}_{\theta}\}_{\theta \in (0,1)}$ as the complex family of interpolation functors, the spaces
$$
X_{k+\theta}:=\mathcal{F}_{\theta}(X_k,X_{k+1}),
\quad
\theta \in (0,1),\ k\in \mathbb{N},
$$
satisfy
$$
X_\alpha=E_q^{\alpha+\frac{N-q}{2q}-\frac{1}{\zeta_g}},
\quad
\alpha\geq 0.
$$
The problem \eqref{problem ns} then takes the abstract form
\begin{align}\label{problem abstract ns}
	\left\{
	\begin{array}{ll}
		u'(t)=\displaystyle\int_{0}^{t} g(t-s)Au(s)\,ds + f(t,u(t)), \quad t>0, \\[0.3em]
		u(0)=u_0,
	\end{array}
	\right.
\end{align}
where $u_0\in X_1$.

In order to prove the $\varepsilon$-regularity condition for the nonlinear term in \eqref{problem abstract ns}, in the next lemma we deal with the convection term in the Navier--Stokes system \eqref{problem ns}.
\begin{lema}\label{Lemma non linearity in NS problem}
	Let $N \geq 3$ and $N/3<q< N$, and suppose that the scalar kernel $g$ in problem \eqref{problem abstract ns} satisfies
	\begin{equation}\label{IntervalonZeta}
		1<\zeta_g< \frac{4q}{N+q}.
	\end{equation}
	Then, for each $\varepsilon>0$ satisfying 
	\begin{equation}\label{IntervalonEpsilon}
		\frac{1}{\zeta_g}-\frac{N}{2q}
		\leq
		\varepsilon
		\leq
		\frac{1}{\zeta_g}-\frac{N+q}{4q},
	\end{equation}
	and with
	$\gamma(\varepsilon)=2\varepsilon+1-\frac{1}{\zeta_g}$,
	there exists a constant $c>0$ such that
	$$
	\|P\left(u\cdot\nabla \right)u-P\left(v\cdot\nabla \right)v\|_{X_{\gamma(\varepsilon)}}
	\leq
	c\,\|u-v\|_{X_{1+\varepsilon}}
	\left(\|u\|_{X_{1+\varepsilon}}+\|v\|_{X_{1+\varepsilon}}\right),
	\quad
	u,v \in X_{1+\varepsilon}.
	$$
	\begin{proof}
		We remark that the condition $q>N/3$ in the statement implies that the interval in \eqref{IntervalonZeta} is nonempty, while the condition $q< N$ implies that the interval in \eqref{IntervalonEpsilon} is nonempty. Moreover, the condition $\zeta_g<\frac{4q}{N+q}$ imposed by \eqref{IntervalonZeta} ensures that the interval in \eqref{IntervalonEpsilon} indeed intersects the interval $\varepsilon>0$.
		
		For the sake of simplicity, let us set $\delta=\frac{N-q}{2q}-\frac{1}{\zeta_g}$, so that $X_\alpha=E^{\alpha-\delta}_q$ for $\alpha\geq 0$. Under the conditions of the statement, we have
		$$
		0\leq 1+\varepsilon-\delta<\frac{N}{2q}
		$$
		and
		$$
		-\frac{N}{2q'}<\gamma(\varepsilon)-\delta\leq 0,
		$$
		which, by the embeddings \eqref{E^a_q Embedding_posivite} and \eqref{E^a_q Embedding_negative}, imply that
		$$
		X_{1+\varepsilon}=E_q^{1+\varepsilon-\delta}\hookrightarrow L^{r_1}(\Omega;\mathbb{R}^N)
		$$
		and
		$$
		X_{\gamma(\varepsilon)}=E_q^{\gamma(\varepsilon)-\delta}\hookleftarrow L^{r_2}(\Omega;\mathbb{R}^N),
		$$
		with
		$$
		r_1=\frac{Nq}{N-2(1+\varepsilon-\delta)q}
		\quad \text{and} \quad
		r_2=\frac{Nq}{N-2(\gamma(\varepsilon)-\delta)q}.
		$$
		Now, let $r_0$ be the exponent satisfying $1/r_0+1/r_1=1/r_2$, that is,
		$$
		r_0=\frac{r_1r_2}{r_1-r_2}=\frac{N}{2(1+\varepsilon-\gamma(\varepsilon))}.
		$$
		Then,
		\begin{align*}
			\|P\left(u\cdot\nabla \right)u-P\left(v\cdot\nabla\right) v\|_{X_{\gamma(\varepsilon)}}
			&\leq c \|(u\cdot \nabla )u - (v \cdot \nabla) v\|_{L^{r_2}} \\
			&= c\|\left((u-v)\cdot \nabla \right)u + \left(v\cdot\nabla \right)(u-v)\|_{L^{r_2}}\\
			&\leq c\|u-v\|_{L^{r_1}}\| u \|_{W^{1,r_0}}
			+ c\|v\|_{L^{r_1}}\| u - v\|_{W^{1,r_0}},
		\end{align*}
		The conditions
		$$
		\frac{2(1+\varepsilon-\delta)}{N}-\frac{1}{q}
		\geq
		\frac{1}{N}-\frac{1}{r_0}
		\quad\text{and}\quad
		q\leq r_0<\infty
		$$
		are satisfied and, by the embedding \eqref{sobolev embeddings}, it follows that
		$$
		W^{2(1+\varepsilon-\delta),q}(\Omega;\mathbb{R}^N)
		\hookrightarrow
		W^{1,r_0}(\Omega;\mathbb{R}^N).
		$$
		Since
		$$
		X_{1+\varepsilon}
		=
		E_q^{1+\varepsilon-\delta}
		\hookrightarrow
		W^{2(1+\varepsilon-\delta),q}(\Omega;\mathbb{R}^N),
		$$
		we obtain
		$$
		X_{1+\varepsilon}\hookrightarrow W^{1,r_0}(\Omega;\mathbb{R}^N).
		$$
		Therefore,
		$$
		\|P\left(u\cdot\nabla \right)u-P\left(v\cdot\nabla\right) v\|_{X_{\gamma(\varepsilon)}}
		\leq
		c\|u-v\|_{X_{1+\varepsilon}}
		\left(\|u\|_{X_{1+\varepsilon}}+\|v\|_{X_{1+\varepsilon}}\right).
		$$
	\end{proof}	
\end{lema}
From Lemma~\ref{Lemma non linearity in NS problem}, it follows that if the map $t \mapsto h(\cdot,t)$ in problem \eqref{problem abstract ns} satisfies a suitable condition near $t=0$, then conditions (H1) and (H2) of Section~\ref{section critical case} on the nonlinearity
$$
f(t,u(t))=-P(u(t)\cdot\nabla u(t))+Ph(\cdot,t)
$$
are satisfied, and we can apply our abstract theorems. This is stated precisely in the next result.

\begin{theorem}\label{theorem well-posedness navier stokes}
	Let $N \geq 3$ and $N/3<q<N$. Let $\{E_q^{\alpha}\}_{\alpha \in \mathbb{R}}$ be the fractional power scale associated with the Stokes operator $A$ on $L^q_\sigma(\Omega;\mathbb{R}^N)$ with domain
	$\mathcal{D}(A)=W^{2,q}(\Omega;\mathbb{R}^N)\cap W^{1,q}_0(\Omega;\mathbb{R}^N) \cap L^q_\sigma(\Omega;\mathbb{R}^N)$.
    Suppose that $g$ satisfies the hypotheses of Theorem~\ref{generation + regularity} -- for instance, if $g$ has the form given in \eqref{kernelappendix} --  with
	\begin{equation*}
		1<\zeta_g< \frac{4q}{N+q}.
	\end{equation*}
	Assume further that, for some $\varepsilon>0$ satisfying
	\begin{equation*}
		\frac{1}{\zeta_g}-\frac{N}{2q}
		\leq
		\varepsilon
		\leq
		\frac{1}{\zeta_g} - \frac{N+q}{4q},
	\end{equation*}
	the function $h:(0,\infty)\to E_q^{2\varepsilon+\frac{N+q}{2q}-\frac{2}{\zeta_g}}$ is measurable and satisfies
	$$
	\|Ph(t)\|_{E_q^{2\varepsilon+\frac{N+q}{2q}-\frac{2}{\zeta_g}}}
	\leq
	\nu(t)t^{-2\zeta_g\varepsilon},
	\quad t>0,
	$$
	where $\nu:(0,\infty)\to \mathbb{R}$ is a non-decreasing function with $\lim_{t \to 0^+}\nu(t)=0$.
	Then, for all $w \in E_q^{\frac{N+q}{2q}-\frac{1}{\zeta_g}}$, there exist $r=r(w)>0$ and $\tau_0=\tau_0(w)>0$ such that problem \eqref{problem abstract ns}
	admits an $\varepsilon$-regular mild solution $u(\,\cdot\,;u_0)$ defined on $[0,\tau_0]$ for all $u_0$ in the closed ball
	$B_{E_q^{\frac{N+q}{2q}-\frac{1}{\zeta_g}}}[w,r]$.
	Moreover, the following statements hold.
	\begin{itemize}
		\item[{(A)}] For all $\theta \in (0,2\varepsilon)$, we have
		$$
		u(\,\cdot\,;u_0)
		\in
		C\!\left((0,\tau_0];E_q^{\theta+\frac{N+q}{2q}-\frac{1}{\zeta_g}}\right),
		$$
		and, if $J \subset B_{E_q^{\frac{N+q}{2q}-\frac{1}{\zeta_g}}}[w,r]$ is compact, then
		$$
		\lim_{t\to0^+}
		t^{\zeta_g\theta}
		\sup_{u_0 \in J}
		\|u(t;u_0)\|_{E_q^{\theta+\frac{N+q}{2q}-\frac{1}{\zeta_g}}}
		=
		0.
		$$
		\item[{(B)}] For each $\theta \in [0,2\varepsilon)$, there exists a constant $L=L(\theta,w)$ such that
		$$
		t^{\zeta_g\theta}
		\|u(t;u_0)-u(t;u_1)\|_{E_q^{\theta+\frac{N+q}{2q}-\frac{1}{\zeta_g}}}
		\leq
		L
		\|u_0-u_1\|_{E_q^{\frac{N+q}{2q}-\frac{1}{\zeta_g}}},
		$$
		for $t\in (0,\tau_0]$ and
		$u_0,u_1 \in B_{E_q^{\frac{N+q}{2q}-\frac{1}{\zeta_g}}}[w,r]$.
		\item[{(C)}] The $\varepsilon$-regular mild solution $u(\cdot;u_0)$ can be continued on an interval $[0,\tau_{\mathrm{max}})$, where $\tau_{\mathrm{max}} \in (\tau_0,\infty]$.
		If $\tau_{\mathrm{max}}<\infty$, then
		$$
		\limsup_{t \to \tau_{\mathrm{max}}^-}
		\|u(t;u_0)\|_{E_q^{\varepsilon+\frac{N+q}{2q}-\frac{1}{\zeta_g}}}
		=
		\infty.
		$$
	\end{itemize}
	Moreover, if $v:[0,\tau_1]\to E_q^{\frac{N+q}{2q}-\frac{1}{\zeta_g}}$ is an $\varepsilon$-regular mild solution of problem \eqref{problem ns} satisfying
	$$
	\lim_{t\to0^+}
	t^{\zeta_g\varepsilon}
	\|v(t)\|_{E_q^{\varepsilon+\frac{N+q}{2q}-\frac{1}{\zeta_g}}}
	=
	0,
	$$
	then $\tau_1<\tau_{\mathrm{max}}$ and $v(t)=u(t;u_0)$ for all $t \in [0,\tau_1]$.
\end{theorem}

\begin{remark}
	Consider the problem \eqref{problem ns}  in power-type materials, that is, taking $g$ satisfying
	$$1\ast g(t)=\frac{t^\alpha}{\Gamma(\alpha+1)}, \quad t>0,$$
	where $\alpha\in[0,1)$. In this case, we can take $\zeta_g=1+\alpha$, see the Appendix for details. This problem was firstly considered by de Andrade, Viana and Silva in \cite{andrade2021viana2021silva}. We remark that their main results on local well-posedness, Theorems 3 and 4,  can be completely recovered by our Theorem \ref{theorem well-posedness navier stokes}. 
\end{remark}

\begin{remark}
	Arrieta and Carvalho prove a similar result of local existence and uniqueness of $\varepsilon$-regular mild solution for the classical Navier-Stokes problem, see \cite[Subsection 3.1]{arrieta2000abstract}. They result can be almost completely recovered by Theorem \ref{theorem well-posedness navier stokes} if we formally replace $\zeta_g=1$, except for the uniqueness, since our result assures uniqueness in the class of $\varepsilon$-regular mild solutions that satisfy
	$$\lim_{t\to0^+} t^{\varepsilon}\|v(t)\|_{E_q^{\frac{N-q}{2q}+\varepsilon}}=0.$$
	Particularly,  our result is compatible with the classical results by Fujita and Kato on the Navier-Stokes equation, see \cite{fujita1963navier,kato1962nonstationary}.
\end{remark}
	%
	%
%
%
%

%
\subsection[Semilinear reaction-diffusion equations with memory]{Semilinear reaction-diffusion equations with memory%
	\sectionmark{Reaction-diffusion equations}}
\sectionmark{Reaction-diffusion equations}

In this section, we study two problems concerning {diffusion phenomena} governed by equations of the form
\begin{equation*}\label{reaction-difussion equation}
	u_t(x,t) =\dint_0^t g(t-s)\Delta u(x,s)\,ds+ f(u(x,t)),\quad x\in \Omega\subset \mathbb{R}^N, \, t>0.
\end{equation*}
Such equations arise naturally in problems of heat flow in homogeneous isotropic rigid heat conductors $\Omega\subset \mathbb{R}^N$ subject to hereditary memory, when the Fourier's constitutive law for the heat flux,
$$q(x,t)=-c \nabla u(x,t),$$
which leads to the classical heat equation, is replaced by
\begin{equation}\label{poq}
	q(x,t)= -\int_0^\infty g(t-s)\nabla u(x,s) ds,
\end{equation}
see, e.g., \cite{monica2014reaction,fabrizio2002asymptotic}.  The model \eqref{poq} for the constitutive law for the flux was proposed by Gurtin and Pipkin in \cite{gurtin1968general}, and thenceforth, it has been well accepted in modeling heat conduction phenomena in materials with memory, see, e.g., \cite{monica2014reaction,hristov2016transient,maccamy1977integro,miller1978integrodifferential,nunziato71onheat}. 
\subsubsection{Super-linear nonlinearity and initial data in Lebesgue spaces}
In recent years, a lot of research about diffusion equations involving a nonlinearity of type $f(u)=c_0|u|^{\rho-1}u$, $c_0 \in \mathbb{R}$, has been made, see, e.g., \cite{brezis1996nonlinear,loayza2006heat,snoussi2001asymptotically,viana2019local,weissler1980,zhang2015blow} and references therein. In this context, we consider $\Omega\subset\mathbb{R}^N$ as being a smooth and bounded domain and study the initial value problem 
\begin{align}\label{Lorentz problem}
	\left\{
	\begin{array}{ll}
		 u_t =\dint_0^t g(t-s)\Delta u(x,s)\,ds+ c_0|u|^{\rho-1}u, \quad \textrm{in } \Omega\times(0,\infty), \\
		u(x,t)=0, \quad (x,t) \in \partial \Omega\times [0,\infty), \\
		u(x,0)=u_0(x), \quad x \in \Omega. \\
	\end{array}\right.
\end{align}
We are focused on determining the critical exponent for the above equation when \(u_0\in L^q(\Omega)\), with  \(1 < q < \infty\).  Let us consider the Laplace operator $L:=\Delta$ with domain $\mathcal{D}(L):=W^{2,q}(\Omega)\cap W^{1,q}_0(\Omega)$. Likewise the Stokes operator, it is well known that $-L:\mathcal{D}(L)\subset L^q(\Omega) \to L^q(\Omega)$ is a sectorial operator of angle $0$, with $0 \in \rho(L)$, and it admits a scale of fractional powers spaces $\{F^q_\beta\}_{\beta \in \mathbb{R}}$ which satisfies
$$F^{(1-\theta)\alpha+\theta\beta}_q=[F^\alpha_q,F^{\beta}_q]_{\theta}, \quad -\infty<\alpha\leq \beta<\infty,\,\theta \in (0,1),$$
and 
\begin{align*}
	F^{\beta}_q \hookrightarrow H^{2\beta}_q(\Omega), \quad \beta \geq 0, \quad 1<q<\infty,
\end{align*}
see \cite[Section 8.3]{haase2006functional} and \cite[Theorem 1.5.4]{amann1995linear}.
Then, as in the previous section, from the embedding \eqref{sobolev embeddings}, we obtain 
\begin{align}\label{wqa}
	\left\{
	\begin{array}{lll}
		F^{\beta}_q \hookrightarrow L^r(\Omega), \quad r=\frac{Nq}{N-2\beta q}, \quad 0\leq \beta<\frac{N}{2q},\\
		F^0_q=L^q(\Omega), \\
		F^{\beta}_q\hookleftarrow L^s(\Omega), \quad s=\frac{Nq}{N-2\beta q}, \quad -\frac{N}{2q'}< \beta \leq 0.
	\end{array}\right.
\end{align}

Further, defining $X_0=F_q^{-1}$ and denoting by $A$ the realization of $L$ in $X_0$, we have that $-A$ is a sectorial operator with domain $\mathcal{D}(A)=L^q(\Omega)\subset X_0$ and, for $k\in \mathbb{N}$, the space $X_k:=\mathcal{D}(A^k)$, $\|x\|_{X_k}:=\|(I-A)^kx\|_{X_0},$ satisfies $X_k:=F^{k-1}_q$. We use again the complex family of interpolation functors $\mathcal{F}_{\theta}=[\cdot,\,\cdot]_{\theta}$ to obtain
$$X_{k+\theta}:=\mathcal{F}_{\theta}(X_k,X_{k+1})=F^{k-1+\theta}_q, \quad \theta \in (0,1), \, k\in \mathbb{N}.$$
In this abstract setting, the problem \eqref{Lorentz problem} becomes
\begin{align}\label{abstract Lorentz problem}
	\left\{
	\begin{array}{ll}
		u'(t)=\dint_{0}^{t}g(t-s)Au(s)ds + f(u(t)), \quad t>0, \\
		u(0)=u_0,
	\end{array}
	\right.
\end{align}
where, $f(w)=c_0|w|^{\rho-1}w$ for $w \in X_1=L^q(\Omega)$.
\begin{lema}\label{lemma non linearity}
	If $a,b,\eta>0$, then
	$$\left|a^\eta - b^\eta\right|\leq \eta |a-b|\left(a^{\eta-1}+b^{\eta-1}\right).$$
	If $\eta>1$, then the same holds with $a,b \geq 0$.
	\begin{proof}
		Let $\phi(t)=t^\eta$, $t>0$. Without loss of generality, we suppose $a<b$. Then,
		\begin{align*}
			|a^\eta - b^\eta| = \left|\int_a^b \phi'(t)dt\right| = \eta\int_a^b t^{\eta-1}dt \leq \eta(b-a)\max_{t \in [a,b]}t^{\eta-1}\\
			\leq \eta|a-b|\left(a^{\eta-1}+b^{\eta-1}\right)	
		\end{align*}
	\end{proof}
\end{lema}
\begin{lema}\label{lemma non linearity 2}
	Let $x,y \in \mathbb{C}$ and $\rho>1$. Then,
	\begin{equation}\label{tyn}
		\left||x|^{\rho-1}x-|y|^{\rho-1}y\right|\leq \rho |x-y|\left(|x|^{\rho-1}+|y|^{\rho-1}\right).
	\end{equation}
	\begin{proof}
		We note that
		\begin{align}\label{plo}
			\left||x|^{\rho-1}x-|y|^{\rho-1}y\right|&= \left|(x-y)|y|^{\rho-1}+x\left(|x|^{\rho-1}-|y|^{\rho-1}\right)\right|\nonumber\\
			&\leq\left|x-y\right||y|^{\rho-1}+|x|\left||x|^{\rho-1}-|y|^{\rho-1}\right|.
		\end{align}
		Clearly, \eqref{tyn} holds if $x=0$ or $y=0$. Then, we can suppose $|x|,|y|>0$ and, applying Lemma \ref{lemma non linearity}, we have
		$$\left||x|^{\rho-1}-|y|^{\rho-1}\right|\leq (\rho-1)\left||x|-|y|\right|\left(|x|^{\rho-2}+|y|^{\rho-2}\right).$$
		Then, by \eqref{plo},
		\begin{align*}
			\left||x|^{\rho-1}x-|y|^{\rho-1}y\right|&\leq |x-y||y|^{\rho-1}+ |x|(\rho-1)\left||x|-|y|\right|\left(|x|^{\rho-2}+|y|^{\rho-2}\right) \\
			&\leq |x-y|\left(|x|^{\rho-1}+|y|^{\rho-1}\right) + (\rho-1)|x-y|\left(|x|^{\rho-1}+|x||y|^{\rho-2}\right).
		\end{align*}
		Without loss of generality, we suppose $|x|\leq |y|$. Then, $|x||y|^{\rho-2}\leq |y|^{\rho-1}$ and so
		\begin{align*}
			\left||x|^{\rho-1}-|y|^{\rho-1}\right|&\leq |x-y|\left(|x|^{\rho-1}+|y|^{\rho-1}\right) + (\rho-1)|x-y|\left(|x|^{\rho-1}+|y|^{\rho-1}\right)\\
			&=\rho|x-y|\left(|x|^{\rho-1}+|y|^{\rho-1}\right).
		\end{align*}
	\end{proof}
\end{lema}
\begin{lema}\label{lemma non linearity 3}
	Let $1<\rho\leq p<\infty$ and $c_0 \in \mathbb{R}$ be fixed, and consider the map $$f:u \mapsto c_0|u|^{\rho-1}u, \quad u\in L^{p}(\Omega).$$
	For $u,v \in L^{p}(\Omega)$, the following inequality holds:
	$$\left\|f(u)-f(v)\right\|_{L^{p/\rho}}\leq c\|u-v\|_{L^{p}}\left(\|u\|_{L^{p}}^{\rho-1}+\|v\|_{L^{p}}^{\rho-1}\right),$$
	where $c=|c_0|\rho$.
	\begin{proof}
		By Lemma \ref{lemma non linearity 2},
		\begin{equation*}
			\left||u|^{\rho-1}u-|v|^{\rho-1}v\right|\leq \rho |u-v|\left(|u|^{\rho-1}+|v|^{\rho-1}\right).
		\end{equation*}
		Then, it follows from Hölder's inequality that
		\begin{align}\label{qlg}
			\left\|f(u)-f(v)\right\|_{L^{p/\rho}}\leq c\|u-v\|_{L^{p}}\left\||u|^{\rho-1}+|v|^{\rho-1}\right\|_{L^s},
		\end{align}
		where $s$ is such that
		$$\frac{1}{p/\rho}=\frac{1}{p}+\frac{1}{s},$$
		that is, $s=\frac{p}{\rho-1}$.
		We note that, for all $w \in L^{p}(\Omega)$,
		$$\left\||w|^{\rho-1}\right\|
		_{L^s}
		=\left(\int_\Omega(|w|^{\rho-1})^s\right)
		^{1/s}=
		\left(\int_\Omega|w|^{p}\right)
		^{\frac{\rho-1}{{p}}}
		=\|w\|^{\rho-1}_{L^{p}}.$$
		Therefore, from \eqref{qlg}, we have
		\begin{align*}
			\left\|f(u)-f(v)\right\|_{L^{p/\rho}}&\leq  c\|u-v\|_{L^{p}}\left(\left\||u|^{\rho-1}\right\|_{L^s}
			+
			\left\||v|^{\rho-1}\right\|_{L^s}\right) \\
			&= c\|u-v\|_{L^{p}}\left(\|u\|_{L^{p}}^{\rho-1}+\|v\|_{L^{p}}^{\rho-1}\right).
		\end{align*}
	\end{proof}
\end{lema}
\begin{lema}
	Let $1<\rho<1+\frac{2q}{N}$, $c_0 \in \mathbb{R}$, and suppose that $1<\zeta_g\leq \frac{2q}{N(\rho-1)}$. Then, the nonlinearity $f(u)=c_0|u|^{\rho-1}u$ satisfies
	$$\|f(u)-f(v)\|_{X_{\gamma(\varep)}}\leq c \|u-v\|_{X_{1+\varep}}\left(\|u\|_{X_{1+\varep}}^{\rho-1}+\|v\|_{X_{1+\varep}}^{\rho-1}\right), \quad u,v \in X_{1+\varep},$$
	for each $\varep$ such that $\max\left\{0,\frac{1}{\rho}\left(\frac{1}{\zeta_g}-\frac{N}{2q'}\right)\right\}<\varep<\min\left\{\frac{1}{\rho\zeta_g},\frac{N}{2q}\right\}$, $\gamma(\varep)=\rho\varep+1-\frac{1}{\zeta_g}$ and some $c>0$.
	\begin{proof}
		Let us denote $p=\frac{Nq}{N-2\varep q}$. By the embeddings \eqref{wqa},
		\begin{equation*}
			X_{1+\varep}=F^{\varep}_q\hookrightarrow L^p(\Omega), \quad  L^{\frac{Nq}{N-2(\gamma(\varep)-1)q}}(\Omega)\hookrightarrow X_{\gamma(\varep)},
		\end{equation*}
		and since
		\begin{align*}
			N-2(\gamma(\varep)-1)q &= N-2\rho\varep q +2q/\zeta_g \\
			& \geq N - 2\rho\varep q + N(\rho-1)\\
			&=(N-2\varep q)\rho \\
			&= (Nq/p)\rho,
		\end{align*}
		we have $$\frac{Nq}{N-2(\gamma(\varep)-1)q}\leq p/\rho.$$
		Then,
		$$L^{p/\rho}(\Omega)\hookrightarrow L^{\frac{Nq}{N-2(\gamma(\varep)-1)q}}(\Omega),$$
		from whence
		$$	X_{1+\varep}\hookrightarrow L^p(\Omega), \quad  L^{p/\rho}(\Omega)\hookrightarrow X_{\gamma(\varep)}
		$$
		and the proof follows from Lemma \ref{lemma non linearity 3}.
	\end{proof}
\end{lema}

Now, we can apply our abstract results to the problem \eqref{abstract Lorentz problem}.

\begin{theorem}
	Let $\{F^\alpha\}_{\alpha \in \mathbb{R}}$ be the fractional power scale associated with the operator $-\Delta$ on $L^q(\Omega)$ with domain $W^{2,q}(\Omega)\cap W^{1,q}_0(\Omega)$ and let $1<\rho<1+\frac{2q}{N}$. Suppose that $g$ satisfies the hypotheses of Theorem~\ref{generation + regularity} -- for instance, if $g$ has the form given in \eqref{kernelappendix} -- with $1<\zeta_g\leq \frac{2q}{N(\rho-1)}$ and let $\varep$ such that $$\max\left\{0,\frac{1}{\rho}\left(\frac{1}{\zeta_g}-\frac{N}{2q'}\right)\right\}<\varep<\min\left\{\frac{1}{\rho\zeta_g},\frac{N}{2q}\right\}.$$
	Then, for all $w \in L^q(\Omega)$, there exist $r=r(w)>0$ and $\tau_0=\tau_0(w)>0$ such that the problem \eqref{abstract Lorentz problem} has an $\varep$-regular mild solution $u(\,\cdot\, ,u_0)$ defined on $[0,\tau_0]$ for all $u_0$ in the closed ball $B_{L^q(\Omega)}[w,r].$ Moreover, the following statements hold.
	\begin{itemize}
		\item[{(A)}] For all $\theta \in (0, \rho\varep)$, we have
		$$u(\,\cdot\,,u_0) \in C\left((0,\tau_0];F^\theta\right)$$
		and, if $J \subset B_{L^q(\Omega)}[w,r]$ is compact, then
		$$\lim_{t\to0^+}t^{\zeta_g\theta}\sup_{u_0 \in J}\|u(t;u_0)\|_{F^\theta}=0.$$
		\item[{(B)}] For each $\theta \in [0, \rho\varep)$, there exists a constant $L=L(\theta,w)$ such that
		$$t^{\zeta_g\theta}\|u(t;u_0)-u(t,u_1)\|_{F^\theta}\leq L \|u_0-u_1\|_{L^q(\Omega)}, \quad t\in (0,\tau_0],\, u_0,u_1 \in B_{L^q(\Omega)}[w,r].$$
		\item[{(C)}] The $\varep$-regular mild solution $u(\cdot;u_0)$ can be continued on an interval $[0,\tau_{\mathrm{max}})$, where $\tau_{\mathrm{max}} \in (\tau_0,\infty]$. If $\tau_{\mathrm{max}}<\infty$, then 
		$$\limsup_{t \to \tau_{\textrm{max}}\,^-} \|u(t;u_0)\|_{_{F^\varep}}=\infty.$$
	\end{itemize}
Moreover, if $v$ is an $\varep$-regular mild solution on some interval $[0,\tau_1]$ for the problem \eqref{abstract Lorentz problem} satisfying
$$\lim_{t\to0^+} t^{\zeta_g\varep}\|v(t)\|_{F^{\varep}}=0,$$
then $\tau_1<\tau_{\mathrm{max}}$ and $v(t) = u(t;u_0)$ for all $t \in [0,\tau_1]$.
\end{theorem}

\begin{remark}
The above result shows us that for $1<\rho<1+\frac{2q}{N}$,  the condition on $\zeta_g$ for the well-posedness of \eqref{Lorentz problem} is
\begin{equation}\label{critico}
	1<\frac{N\zeta_g(\rho-1)}{2}\leq q,
\end{equation}
underscoring the influence of the material function in the analysis of the problem. For example, in the context of a Hookean Solid, that is, taking $g$ satisfying
	$$1\ast g(t)=\mu t, \quad t>0,$$
for some $\mu\ne0$, we can choose $\zeta_g=2$, see the Appendix for details, and then \eqref{critico} reduces to $	1<N(\rho-1)\leq q$. On the other hand, in the context of power-type materials we have
$$1\ast g(t)= \frac{t^{\alpha}}{\Gamma(\alpha+1)}, \quad t>0,$$
	where $\Gamma$ is the Gamma function and $\alpha\in(0,1)$. Hence $\zeta_g=1+\alpha$, and \eqref{critico} takes the form
\begin{equation}\label{criticofrac}	
	1<\frac{N(\alpha+1)(\rho-1)}{2}\leq q.
\end{equation}	
\end{remark}

\begin{remark}
 At this point, it is important to observe that \eqref{critico} is closely related with the critical case for the nonlinear  heat equation
\begin{equation*}
	\left\{ \begin{array}{ll}
		u_{t}  = \Delta u +|u|^{\rho-1}u,~\mbox{in}~[0,\infty)\times \Omega,\\
		u=0,~\mbox{on}~[0,\infty)\times \partial\Omega,~~\\
		u(0,x) = u_0(x),~\mbox{in}~\Omega,&
	\end{array} \right.
\end{equation*}
when $u_0\in L^{q}(\Omega)$.  Indeed, Brezis and Cazenave \cite{brezis1996nonlinear} have proved that the value
\begin{equation}\label{criticoheat}
	q=\frac{N(\rho-1)}{2}
\end{equation}
 plays a critical role to the above problem, see also \cite{arrieta2000abstract, weissler1980}. In \cite{andrade2015}, de Andrade et al. show that this same value is the critical exponent for the subdiffusive heat equation  
\begin{equation*}
	\left\{ \begin{array}{ll}
		\partial_{t}^{\gamma} u = \Delta u + |u|^{\rho-1}u,~\mbox{in}~[0,\infty)\times \Omega,\\
		u=0,~\mbox{on}~[0,\infty)\times \partial\Omega,~~\\
		u(0,x) = u_0(x),~\mbox{in}~\Omega,&
	\end{array} \right.
\end{equation*}
where $\partial_{t}^{\gamma}$ is Caputo's fractional derivative,  $\gamma \in (0,1)$, and $u_0\in L^{q}(\Omega)$. Likewise,  Costa et al. \cite{costa2025} recently proved that \eqref{criticoheat} is also the critical exponent for the diffusion-wave equation
\begin{equation}\label{diffwave}
	\left\{ \begin{array}{ll}
		\partial_{t}^{\alpha} u = \Delta u + |u|^{\rho-1}u,~\mbox{in}~[0,\infty)\times \Omega,\\
		u=0,~\mbox{on}~[0,\infty)\times \partial\Omega,~~\\
		u(0,x) = u_0(x),~ u'(0,x) = u_1(x),~\mbox{in}~\Omega,&
	\end{array} \right.
\end{equation}
where $\partial_{t}^{\alpha}$ is Caputo's fractional derivative,  $\alpha \in (1, 2)$, and $u_0,u_1\in L^{q}(\Omega)$.  These facts highlight the compatibility of our results with the previous ones.
\end{remark}

\begin{remark}
In the context of power-type materials, the integrodifferential equation \eqref{Lorentz problem} is given by
\begin{align}\label{intpower}
	\left\{
	\begin{array}{ll}
		u_t =\dint_0^t \frac{(t-s)^{\alpha-1}}{\Gamma(\alpha)}\Delta u(x,s)\,ds+ |u|^{\rho-1}u, \quad \textrm{in } \Omega\times(0,\infty), \\
		u(x,t)=0, \quad (x,t) \in \partial \Omega\times [0,\infty), \\
		u(x,0)=u_0(x), \quad x \in \Omega,
	\end{array}\right.
\end{align}
where $\Gamma$ is the Gamma function and $\alpha\in[0,1)$. Note that the linear part of this equation is equivalent\footnote{These linear problems share the same resolvent family.} to the linear part of the fractional diffusion-wave equation \eqref{diffwave}, if we take $u_1\equiv 0$. However, these nonlinear problems have different critical exponents. In fact, as pointed out by the above remarks, the critical exponent of \eqref{intpower} is characterized by  \eqref{criticofrac}, and the critical exponent of \eqref{diffwave} is given by \eqref{criticoheat}. This fact occurs because the non-local nature of the fractional derivative affects both the linear and nonlinear parts of \eqref{diffwave} equally. Whereas in \eqref{intpower} the non-local effect is present only in the linear part of the problem. Even so, it is interesting to note that
$$\lim_{\alpha\to 0}\frac{N(\alpha+1)(\rho-1)}{2}=\frac{N(\rho-1)}{2}\quad and \quad \lim_{\alpha\to 1}\frac{N(\alpha+1)(\rho-1)}{2}=N(\rho-1).$$
\end{remark}

\subsubsection{Gradient nonlinearity and initial data in Besov spaces}
 We close this section of applications studying the problem
\begin{align}\label{Besov problem}
	\left\{
	\begin{array}{ll}
		u_t(x,t) =\dint_0^t g(t-s)\Delta u(x,s)\,ds+ c_0\|\nabla u(x,t)\|^\rho, \quad (x,t)\in \mathbb{R}^N\times(0,\infty), \\
		u(x,0)=u_0(x), \quad x \in \mathbb{R}^N, \\
	\end{array}\right.
\end{align}
with initial data in the {Besov space} $B^s_{p,q}(\mathbb{R}^N)$, where $s \in \mathbb{R}$, $1\leq p,  q\leq \infty$. From a mathematical perspective, this equation is a variant of the viscous Hamilton-Jacobi equation that incorporates memory effects. It is worth noting that there is already extensive research on diffusion equations with a gradient-type nonlinearity. We recommend the references \cite{almeida2015viana, chen2014semilinear, chipot1989some, gilding2005cauchy, kaltenbacher2019jordan, souplet1996finite} for further reading. However, as far as we know, this work is the first to deal with problem \eqref{Besov problem} in this context.

Before proceeding, we summarize some useful properties about the Besov spaces as follows:

\begin{itemize}
	\item[(a)] For $s \in \mathbb{R}$ and $1\leq p,q\leq \infty$, the operator $(I-\Delta):B^{s+2}_{p,q}(\mathbb{R}^N)\to B^s_{p,q}(\mathbb{R}^N)$ is an isomorphism.
	\item[(b)] Let $1\leq p,q\leq \infty$, $m \in \mathbb{N}$ and $s \in \mathbb{R}$. Then the equivalence of norms
	$$\|\cdot\|_{B^{s+m}_{p,q}} \approx \|\cdot\|_{B^{s}_{p,q}} + \sum_{j =1}^N \left\|\frac{\partial^m}{\partial x_j}\,\cdot\right\|_{B^{s}_{p,q}}$$
	holds.
	\item[(c)] Let $1\leq p \leq \infty$. Then, $$B^0_{p,1}(\mathbb{R}^N)\hookrightarrow L^p(\mathbb{R}^N) \hookrightarrow B^0_{p,\infty}(\mathbb{R}^N).$$
	\item[(d)] Let $1\leq p,q_1,q_2\leq \infty$, $s \in \mathbb{R}$ and $\delta>0$. Then
	$$B^{s+\delta}_{p,q_1}(\mathbb{R}^N)\hookrightarrow B^s_{p,q_2}(\mathbb{R}^N).$$
	\item[(e)] Let $1 \leq \widetilde{p} < p \leq \infty$, $1\leq q \leq \infty$ and $-\infty < s_1 \leq s_0 < \infty$ satisfying $s_0-{N}/{\widetilde{p}}= s_1 - {N}/{p}$. Then, $$B^{s_0}_{\widetilde{p},q}(\mathbb{R}^N) \hookrightarrow B^{s_1}_{p,q}(\mathbb{R}^N).$$
	\item[(f)] Let $1\leq p,q,q_0,q_1$, $-\infty<s_0<s_1<\infty$ and, for some $\theta \in (0,1)$, $s=(1-\theta)s_0+\theta s_1$. Then $$B^s_{p,q}(\mathbb{R}^N)=\left(B^{s_0}_{p,q_0}(\mathbb{R}^N),B^{s_1}_{p,q_1}(\mathbb{R}^N)\right)_{\theta,q}$$
	with equivalent norms.
	\item[(g)] For $1\leq p,q\leq \infty$ and $s \in \mathbb{R}$, the operator $-\Delta$ with domain $$\mathcal{D}(-\Delta)=B^{s+2}_{p,q}(\mathbb{R}^N)\subset B^s_{p,q}(\mathbb{R}^N)$$ has spectrum $\sigma(-\Delta)=[0,\infty)$ and is sectorial of angle $0$ on $B^s_{p,q}(\mathbb{R}^N)$.
	\item[(h)]Let $1\leq p \leq \infty$. Then, $$B^1_{p,1}(\mathbb{R}^N)\hookrightarrow W^{1,p}(\mathbb{R}^N) \hookrightarrow B^1_{p,\infty}(\mathbb{R}^N).$$
\end{itemize} 
The proof of Items (a)-(g) can be founded in \cite{sawano2018theory},  while (h) is a consequence of (b) and (c).

Now, for $s \in \mathbb{R}$, let us consider the operator $A:=\Delta$ on $X_0:=B^{-2+s}_{p,q}(\mathbb{R}^N)$, with domain $$\mathcal{D}(A):=B^{s}_{p,q}(\mathbb{R}^N)\subset X_0,$$ 
and  $X_k:=\mathcal{D}(A^k)$, with norm given by $\|x\|_{X_k}:=\|(I-A)^kx\|_{X_0}$, $k=1,2$. Then $-A$ is sectorial operator of angle $0$; by real interpolation we define
$$X_{k+\theta}=\left(X_{k},X_{k+1}\right)_{\theta,q}, \quad \theta \in (0,1),\, k \in \mathbb{N},$$
and hence
$$X_{\alpha}=B^{2(\alpha-1)+s}_{p,q}(\mathbb{R}^N), \quad \alpha \in [0,2].$$
In this formulation, \eqref{Besov problem} is rewritten as
\begin{align}\label{abstract Besov problem}
	\left\{
	\begin{array}{ll}
		u'(t)=\dint_{0}^{t}g(t-s)Au(s)ds + f(u(t)), \quad t>0, \\
		u(0)=u_0 \in X_1,
	\end{array}
	\right.
\end{align}
where $f(u):=c_0\|\nabla u\|^\rho$.

We start estimating the behavior of the nonlinearity $f$ on the Besov spaces scale.
\begin{lema}\label{lemma f sobolev-lebesgue in besov problem}
	Let $1<\rho<p<\infty$ and $c_0 \in \mathbb{R}$. Then, $f:W^{1,p}(\mathbb{R}^N)\to L^{\frac{p}{\rho}}(\mathbb{R}^N)$ is well-defined and satisfies
	$$\|f(u)-f(v)\|_{L^{p/\rho}}\leq c\|u-v\|_{W^{1,p}}\left(\|u\|_{W^{1,p}}^{\rho-1}+\|v\|_{W^{1,p}}^{\rho-1}\right), \quad u,v \in W^{1,p}(\mathbb{R}^N),$$
	where $c=\rho|c_0|$.
	\begin{proof}
		Let $u,v \in W^{1,p}(\mathbb{R}^N)$. By Lemma \ref{lemma non linearity}, for all $x \in \mathbb{R}^N$, we have
		\begin{align*}
			|\|\nabla u(x)\|^\rho-\|\nabla v(x)\|^\rho| &\leq \rho |\|\nabla u(x)\|-\|\nabla v(x)\|| \left(\|\nabla u(x)\|^{\rho-1}+\|\nabla v(x)\|^{\rho-1}\right)\\
			& \leq \rho \|\nabla u(x)-\nabla v(x)\| \left(\|\nabla u(x)\|^{\rho-1}+\|\nabla v(x)\|^{\rho-1}\right).
		\end{align*}
		Then, by Hölder's inequality,
		\begin{align*}
			\|f(u)-f(v)\|_{L^{p/\rho}} &\leq c \|\nabla u-\nabla v\|_{L^{p}} \left\|\|\nabla u\|^{\rho-1}+\|\nabla v\|^{\rho-1}\right\|_{L^{\frac{p}{\rho-1}}} \\
			& \leq c \|\nabla u-\nabla v\|_{L^{p}} \left\|\|\nabla u\|^{\rho-1}\right\|_{L^{\frac{p}{\rho-1}}}+\left\|\|\nabla v\|^{\rho-1}\right\|_{L^{\frac{p}{\rho-1}}} \\
			& =c\|\nabla u-\nabla v\|_{L^{p}} \left(\left\|\nabla u\right\|_{L^p}^{\rho-1}+\left\|\nabla v\right\|_{L^p}^{\rho-1}\right) \\
			&\leq c \|u- v\|_{W^{1,p}} \left(\left\| u\right\|_{W^{1,p}}^{\rho-1}+\left\| v\right\|_{W^{1,p}}^{\rho-1}\right).
		\end{align*}
	\end{proof}
\end{lema}

Next, we apply Lemma \ref{lemma f sobolev-lebesgue in besov problem} and suitable Besov embeddings to derive the necessary estimates on $f$.

\begin{lema}
	Let $1<p<\infty$, $1\leq q \leq \infty$, $1<\rho<\min\left\{p,1+\frac{p}{N}\right\}$, $1-\left(\frac{1}{\rho-1}-\frac{N}{p}\right)<s\leq1$ and
	$$X_{\alpha}:=B^{2(\alpha-1)+s}_{p,q}\left(\mathbb{R}^N\right), \quad \alpha \in [0,2].$$
	If $1<\zeta_g<2/\chi,$  where $\chi(s,p,\rho) =1+\left(1-s+N/p\right)\left(\rho-1\right)$, then the map $f:X_{1+\varep} \to X_{\gamma(\varep)}$ is well defined and satisfies

	$$\|f(u)-f(v)\|_{X_{\gamma(\varep)}}\leq c\|u-v\|_{X_{1+\varep}}\left(\|u\|_{X_{1+\varep}}^{\rho-1}+\|v\|_{X_{1+\varep}}^{\rho-1}\right), \quad u,v \in X_{1+\varep},$$
	for each $\varep$ such that
	\begin{equation}\label{qju}
		\frac{1-s}{2}<\varep<\frac{1-s}{2}+\frac{1}{\rho}\left(\frac{1}{\zeta_g}-\frac{\chi}{2}\right),
	\end{equation}
	and $\gamma(\varep)=\rho\varep+1-1/\zeta_g$.
	\begin{proof}
		From the first inequality in \eqref{qju}, we have $2\varep+s>1$. Then,
		\begin{equation}\label{ape}
			X_{1+\varep}=B^{2\varep+s}_{p,q}(\mathbb{R}^N) \hookrightarrow B^1_{p,1}(\mathbb{R}^N)\hookrightarrow W^{1,p}(\mathbb{R}^N).
		\end{equation}
		Now, we desire to prove that
		\begin{equation}\label{apc}
			L^{\frac{p}{\rho}}(\mathbb{R}^N)\hookrightarrow X_{\gamma(\varep)}.
		\end{equation}
       By the second inequality in \eqref{qju}, we get
		%
		\begin{align}\label{lap}
			\gamma(\varep)&<\rho\left(\frac{1-s}{2}+\frac{1}{\rho}\left(\frac{1}{\zeta_g}-\frac{\chi}{2}\right)\right)+1-\frac{1}{\zeta_g} \nonumber\\
			&= \frac{\rho(1-s)}{2}-\frac{\chi}{2}+1 \nonumber\\
			& = \frac{\rho(1-s)}{2}-\frac{ 1+\left(1-s+N/p\right)\left(\rho-1\right)}{2}+1 \nonumber\\
			& = 1 - \frac{s}{2}-\frac{N(\rho-1)}{2 p} < 1. 
		\end{align}
		In particular, $\gamma(\varep)$ satisfies $\rho\varep+1-\frac{1}{\zeta_g}\leq \gamma(\varep)<1$. To prove \eqref{apc}, first, we note that
		$$L^{\widetilde{p}}(\mathbb{R}^N) \hookrightarrow B^0_{\widetilde{p},\infty}\hookrightarrow B^r_{\widetilde{p},\infty}(\mathbb{R}^N),$$
		where $\widetilde{p}=\frac{p}{\rho}$ and $r=\frac{N}{p}-\frac{N}{\widetilde{p}}=-\frac{N(\rho-1)}{p}$. Further, by \eqref{lap},
		\begin{align*}
			2\left(\gamma(\varep)-1\right)+s < 2 \left(-\frac{s}{2}-\frac{N(\rho-1)}{2 p}\right)+s  = -\frac{N(\rho-1)}{p} = r.
		\end{align*}
		Then,
		\begin{equation*}
			B^r_{\widetilde{p},\infty}(\mathbb{R}^N) \hookrightarrow B^{2(\gamma(\varep)-1)+s}_{p,q}(\mathbb{R}^N)=X_{\gamma
				(\varep)},
		\end{equation*}
and, consequently,
		\begin{equation*}
			L^{\widetilde{p}}(\mathbb{R}^N)\hookrightarrow X_{\gamma(\varep)}.
		\end{equation*}
	The statement follows from \eqref{ape}, \eqref{apc}, and Lemma \ref{lemma f sobolev-lebesgue in besov problem}.
	\end{proof}
\end{lema}

As consequence of our abstract results we have the following theorem.

\begin{theorem}\label{viscous}
	Let $1<p<\infty$, $1\leq q \leq \infty$, $1<\rho<\min\left\{p,1+\frac{p}{N}\right\}$, $1-\left(\frac{1}{\rho-1}-\frac{N}{p}\right)<s\leq1$ and suppose that $g$ satisfies the hypotheses of Theorem~\ref{generation + regularity} -- for instance, if $g$ has the form given in \eqref{kernelappendix} -- with $1<\zeta_g<2/\chi,$  where $\chi =1+\left(1-s+N/p\right)\left(\rho-1\right)$. Also let $\varep$ with
	\begin{equation*}
		\frac{1-s}{2}<\varep<\frac{1-s}{2}+\frac{1}{\rho}\left(\frac{1}{\zeta_g}-\frac{\chi}{2}\right).
	\end{equation*}
	Then, for all $w \in B^s_{p,q}(\mathbb{R}^N)$, there exist $r=r(w)>0$ and $\tau_0=\tau_0(w)>0$ such that the problem \eqref{abstract Besov problem} has a mild solution $u(\,\cdot\, ,u_0)$ defined on $[0,\tau_0]$ for all $u_0$ in the closed ball $B_{B^s_{p,q}}[w,r].$ Moreover, the following statements hold.
	\begin{itemize}
		\item[{(A)}] For all $\theta \in (0, \rho\varep)$, we have
		$$u(\,\cdot\,,u_0) \in C\left((0,\tau_0];B^{2\theta+s}_{p,q}(\mathbb{R}^N)\right)$$
		and, if $J \subset B_{B^{s}_{p,q}}[w,r]$ is compact, then
		$$\lim_{t\to0^+}t^{\zeta_g\theta}\sup_{u_0 \in J}\|u(t;u_0)\|_{B^{2\theta+s}_{p,q}}=0.$$
		\item[{(B)}] For each $\theta \in [0, \rho\varep)$, there exists a constant $L=L(\theta,w)$ such that
		$$t^{\zeta_g\theta}\|u(t;u_0)-u(t,u_1)\|_{B^{2\theta+s}_{p,q}}\leq L \|u_0-u_1\|_{B^{s}_{p,q}},$$
		for $t\in (0,\tau_0],\, u_0,u_1 \in B_{B^{s}_{p,q}}[w,r].$
		\item[{(C)}] The $\varep$-regular mild solution $u(\cdot;u_0)$ can be continued on an interval $[0,\tau_{\mathrm{max}})$, where $\tau_{\mathrm{max}} \in (\tau_0,\infty]$. If $\tau_{\mathrm{max}}<\infty$, then 
		$$\limsup_{t \to \tau_{\textrm{max}}\,^-} \|u(t;u_0)\|_{_{B^{2\varep+s}_{p,q}}}=\infty.$$
	\end{itemize}
Moreover, if $v$ is an $\varep$-regular mild solution on some interval $[0,\tau_1]$ for the problem \eqref{abstract Besov problem} satisfying
$$\lim_{t\to0^+} t^{\zeta_g\varep}\|v(t)\|_{_{B^{2\varep+s}_{p,q}}}=0,$$
then $\tau_1<\tau_{\mathrm{max}}$ and $v(t) = u(t;u_0)$ for all $t \in [0,\tau_1]$.
\end{theorem}

\begin{remark}
In \cite{BSW}, Ben-Artzi, Souplet and Weissler consider the viscous Hamilton-Jabobi equation 
$$
	\left\{
\begin{array}{ll}
	u_t(x,t) =\Delta u(x,t)+ c_0\|\nabla u(x,t)\|^\rho, \quad (x,t)\in \mathbb{R}^N\times(0,\infty), \\
	u(x,0)=u_0(x), \quad x \in \mathbb{R}^N, \\
\end{array}\right.
$$
proving that, among other things, for initial data $u_0\in L^{p}(\mathbb{R}^N)$, $1\leq p < \infty$,  the value
$$p_{c}=\frac{N(\rho-1)}{2-\rho}$$
plays a critical role for the problem, when $\rho<2$. Indeed, the problem is well-posed when $p\ge p_{c}$, and ill-posed if $1\le p< p_{c}$. 

Note that if $s=0$ we have
$$B^0_{p,1}(\mathbb{R}^N)\hookrightarrow L^p(\mathbb{R}^N) \hookrightarrow B^0_{p,\infty}(\mathbb{R}^N),$$
for any $1\leq p \leq \infty$; this suggest that the spaces $B^0_{p,q}(\mathbb{R}^N)$ is close to $L^{p}(\mathbb{R}^N)$, $1\leq p,q \leq \infty$.
Furthermore, taking $s=0$ in Theorem \ref{viscous}, the condition on $\zeta_g$ for the well-posedness of \eqref{Besov problem} becomes $1<\zeta_g<2/\chi,$  where $\chi =1+\left(1+N/p\right)\left(\rho-1\right)$, that is,
$$p>\frac{N\zeta_g(\rho-1)}{2-\rho\zeta_g},$$
highlighting the influence of the material function in the analysis of the problem. We recover $p_{c}$ if we formally take $\zeta_g=1$.

\end{remark}

\section*{Appendix}
As the problems presented in Section \ref{chapter aplications}, many other can be formulated in the abstract format
$$u'(t)=\int_0^t g(t-s)Au(s)ds+f(t,u(t)), \quad t>0, \quad u(0)=u_0.$$
A rich source of them is the continuum mechanics for {materials with memory}, i.e., the theory of viscoelastic materials. In such problems, the scalar kernel $g$, also called {material function}, is associated with the physical properties\footnote{More precisely, to the {shear modulus}.} of the type of material under consideration. With these facts in mind, the prototype of material function we consider has the form \eqref{kernelappendix}, that is, 
\begin{equation*}
	g(t)=\sum_{i=1}^n k_it^{\alpha_i -1}e^{c_i t}, \quad t>0,
\end{equation*}
where $k_i>0, \alpha_i>0, c_i \in \mathbb{R}.$ We mention that many types of materials, as the {Hookean solid}, the {Maxwell fluid}, the {Poynting-Thompson solid}, and the {power type materials}, are described by this kind of material functions, see \cite[Chapter 5]{pruss2013evolutionary} with $a(t)=(1\ast g)(t)$, $t>0$.

In order to apply Theorem \ref{generation + regularity}, such a function $g$ must satisfies the following conditions:
\begin{itemize}
	\item[(B1)] $g \in L^1_{loc}([0,\infty);\mathbb{C})$ is a non identically zero Laplace transformable function;
	\item[(B2)] $\widehat{g}(\cdot)$ admits meromorphic extension to some sector $\Sigma(\omega_0,\eta_0+\pi/2)$ with $\omega_0 \geq 0$ and $\eta_0 \in (0,\pi/2]$, and $\widehat{g}(\lambda)\neq 0$ for all $\lambda \in \Sigma(\omega_0,\eta_0+\pi/2)$;
	\item[(B3)] for each $\omega_1>\omega_0$ and $\eta_1 \in (0,\eta_0)$, there exists $\psi_1 \in (\psi_0,\pi/2)$ such that
	$$-\lambda / \widehat{g}(\lambda) \in\mathbb{C} \backslash \Sigma[0,\psi_1], \quad \textrm{for all }\lambda \in \Sigma(\omega_1,\eta_1+\pi/2);$$
	\item[(B4)] for some $\omega>\omega_0$ and $\eta \in (0,\eta_0)$, there is $\zeta_g>1$ such that
	$$\limsup_{\substack{|\lambda|\to\infty, \\ \lambda \in \Sigma(\omega,\eta+\pi/2)}}\frac{1}{|\widehat{g}(\lambda)||\lambda|^{\zeta_g-1}}<\infty.$$
\end{itemize}

Note that if $g$ is as before, with $k_i>0$, $\alpha_i>0$, $c_i \in \mathbb{R}$, then
$$\widehat{g}(\lambda)=\sum_{i=1}^n \frac{k_i\Gamma(\alpha_i)}{(\lambda-c_i)^{\alpha_i}},$$
where $\Gamma(\cdot)$ denotes the {gamma function}. We claim that, if
	$$\alpha:=\max_i\{\alpha_{i}\} \leq 1 - \frac{\psi_0}{\pi/2},$$
	then (B1)-(B4) are satisfied with
	$$\eta_0:=\frac{(1-\alpha)\frac{\pi}{2}-\psi_0}{(1+\alpha)}, \quad \omega_0:=\max\{0,c_1,...,c_n\}\quad \mbox{and} \quad \zeta_g:=1+\min_{i}\{\alpha_{i}\}.$$
Indeed, to prove that, for $\omega_1>\omega_0$ and $\eta_1 \in (0,\eta_0)$, let $\psi_1$ be defined by the equality
$$\eta_1=\frac{(1-\alpha)\frac{\pi}{2}-\psi_1}{(1+\alpha)},$$
that is 
$$\psi_1:=\pi - (\alpha+1)(\eta_1+\pi/2).$$
Then, $\psi_1 \in (\psi_0,\pi/2)$ and for each $\lambda \in \Sigma(\omega_1,\eta_1+\pi/2)$ and $i=1,...,n$, we have
\begin{align*}
	&\lambda - c_i \in \Sigma(\omega_1-c_i,\eta_1+\pi/2) \subset \Sigma(0,\eta_1+\pi/2) \\
	\Rightarrow \quad& \frac{1}{\lambda-c_i} \in  \Sigma(0,\eta_1+\pi/2) \\ 
	\Rightarrow \quad& \frac{1}{(\lambda-c_i)^{\alpha_i}} \in \Sigma(0,\alpha_i(\eta_1+\pi/2))\subset \Sigma(0,\alpha(\eta_1+\pi/2))\\
	\Rightarrow \quad& \frac{k_i \Gamma(\alpha_i)}{(\lambda-c_i)^{\alpha_i}} \in \Sigma(0,\alpha(\eta_1+\pi/2)).
\end{align*}
Hence,
\begin{align*}
	&\widehat{g}(\lambda)= \sum_{i=1}^k \frac{k_i \Gamma(\alpha_i)}{(\lambda-c_i)^{\alpha_i}} \in \Sigma(0,\alpha(\eta_1+\pi/2)) \in \Sigma(0,\alpha(\eta_1+\pi/2))\\
	\Rightarrow \quad & 1/\widehat{g}(\lambda) \in \Sigma(0,\alpha(\eta_1+\pi/2))\\
	\Rightarrow \quad & \lambda(1/\widehat{g}(\lambda)) \in \Sigma(0,(\alpha+1)(\eta_1+\pi/2))\\
	\Rightarrow \quad&- \lambda/\widehat{g}(\lambda) \in \mathbb{C} \backslash \Sigma[0,\pi - (\alpha+1)(\eta_1+\pi/2)]=\mathbb{C}\backslash \Sigma[0,\psi_1],
\end{align*}
what checks (B3). Finally, let $\zeta_g=1+\min_{i}\{\alpha_{i}\}$ and let $\{i_1<...<i_r\}$ be the set of all indexes $i$ such that $\alpha_{i}$ reaches the minimum between $\alpha_1,...,\alpha_n$. Then, $\alpha_{i_1}=...=\alpha_{i_r}=\zeta_g-1$ and $\alpha_{i}>\zeta_g-1$, for $i \notin \{i_1<...<i_r\}$, from whence
\begin{align*}
	|\widehat{g}(\lambda)||\lambda|^{\zeta_g-1}&=\left|\sum_{i=1}^n k_i\Gamma(\alpha_i)\frac{\lambda^{\zeta_g-1}}{(\lambda-c_i)^{\alpha_i}}\right| \to \sum_{s=1}^r k_{i_s}\Gamma(\alpha_{i_s})>0, \quad \textrm{as }|\lambda|\to \infty,
\end{align*}
and (B4) holds.

\noindent{\bf Declarations:} 
\begin{itemize}
	\item	 Ethics approval and consent to participate: Not applicable.
	\item	Consent for publication: Not applicable.
	\item	Availability of data and materials: Not applicable.
	\item	Competing interests: The authors have no relevant financial or non-financial interests to disclose.
	\item	Funding: Bruno de Andrade is partially supported by CNPQ/Brazil (grant 310384/2022-2) and FAPITEC/Sergipe/Brazil (grant 019203.01303/2024-1).
	\item	Authors' contributions:  B. de Andrade contributed to the implementation of the research, the analysis of the results, and the writing of the manuscript. M. G. Santana contributed to the implementation of the research, the analysis of the results, and the writing of the manuscript.
	\item	Acknowledgements:  Not applicable.
\end{itemize}



\end{document}